\documentclass[11pt,a4paper]{article}

\usepackage{amsmath,amssymb,amscd,amsthm,amsfonts}
\usepackage{amsmath,amsfonts,amssymb,amsthm,epsfig,epstopdf,titling,url,array, amscd}
\usepackage{dsfont}
\usepackage{tikz}
\usepackage[margin=3cm]{geometry}
\usepackage{floatrow}
\usepackage{hyperref}
\usepackage[nobysame, alphabetic]{amsrefs}

\theoremstyle{plain}
\newtheorem{theorem}{Theorem}[section]

\newtheorem{proposition}{Proposition}

\theoremstyle{definition}
\newtheorem{definition}{Definition}
\newtheorem{conjecture}{Conjecture}
\newtheorem{example}{Example}

\newcommand{\rr}{\mathds{R}}
\newcommand{\conv}{\operatorname{conv}}
\newcommand{\tv}{\operatorname{Tv}}
\newcommand{\diam}{\operatorname{diam}}

\author{Deborah Oliveros\thanks{Oliveros and Torres's research is funded by UNAM PAPIIT grant AGI00721.}, \'Erika Rold\'an, Pablo Sober\'on\thanks{Sober\'on's research is funded by grants NSF DMS 2054419 and NSF DMS 2237324}, Antonio J.\ Torres.}
\title{Tverberg Partition Graphs}

\begin{document}

\maketitle

\begin{abstract}
Given a finite set of points in $\rr^d$, Tverberg's theorem guarantees the existence of partitions of this set into parts whose convex hulls intersect.  We introduce a graph structured on the family of Tverberg partitions of a given set of points, whose edges describe closeness between different Tverberg partitions.  We prove bounds on the minimum and maximum degree of this graph, the number of vertices of maximal degree, its clique number, and its connectedness.
\end{abstract}



\section{Introduction}

The study of intersection properties of convex sets is a central topic in discrete geometry.  One of the most important results in this area is Tverberg's theorem, proved in 1966.  Given a sufficiently large set of points in $\rr^d$, Tverberg's theorem guarantees the existence of partitions into $r$ subsets whose convex hulls intersect.  Such partitions are called Tverberg partitions.

\begin{theorem}[Tverberg's Theorem, 1966 \cite{tverberg1966generalization}]
Let $r$ and $d$ be positive integers, and $S\subset \rr^d$ be a set with at least ${\tv(d,r)=(d + 1)(r - 1) + 1}$ points. Then there exists a partition of $S$ into $r$ parts whose convex hulls intersect.
\end{theorem}

In addition to its multiple applications in discrete geometry, Tverberg's theorem showcases the connections of this area with topology and combinatorics.  Research around Tverberg's theorem remains active and it has inspired multiple generalizations \cite{Barany:2018fy, de2019discrete}. Nevertheless, there are relatively few structural results regarding the set of all Tverberg partitions of a given set.  

Given a set $S$ of $n$ points in $\rr^d$, consider the set of all its Tverberg partitions into $r$ parts.  Tverberg's theorem and some of its generalizations guarantee that this set is not empty for $n$ sufficiently large, but not much more is known about it.  For example, even finding optimal lower bounds for the number 
of Tverberg partitions when $|S|=(d+1)(r-1)+1$ is a longstanding open problem that is known as Sierksma's conjecture \cite{Sierksma1979}. This conjecture claims that if $|S|=(d+1)(r-1)+1$, there always exist at least $(r-1)!^d$ Tverberg partitions of $S$.  The best known lower bound on the number of Tverberg partitions, under the assumption of $r$ being a prime number, is $\frac{1}{(r-1)!}(r/2)^{(r-1)(d+1)/2}$  \cite{Vucic:1993be}, which is roughly the square root of Sierkma's conjecture.

In this manuscript, we study structural parameters of the family of Tverberg partitions of a given set $S \subset \rr^d$ on different configurations of points, including when $|S| > \tv(d,r)$. To accomplish this, we introduce a graph that has as vertices the set of all Tverberg partitions of a given set $S$, and we measure edge connectivity between Tverberg partitions by observing how ``\emph{different}'' they are. Our aim is to take the first steps to understanding the structural properties of this graph, such as connectivity, diameter, and degree distribution. In particular, we will analyze graphs associated with Radon partitions, that is, when $r=2$ and $d\geq 2$, in more detail. 

The precise construction of this graph is as follows. % Given two Tverberg partitions $P$ and $P'$ of a set $S$, we include an edge between their vertices if we can change a single point of $P$ to another part in the partition and obtain $P'$.  We denote this graph by $G_T[d,r](S)$, or simply $G_T[d,r]$ if there is no risk of confusion.

\begin{definition}\label{def-TverbergGraph}
Given a set of points $S$ in $\rr^d$ and a positive integer $r$, we define the \emph{Tverberg $r$-partition graph of $S$}, denoted by $G_{T}[S,r]$, as the graph with one vertex for each Tverberg partition of $S$ into $r$ parts.  Two of these partitions $P$ and $P'$ are connected by an edge if and only if there exists exactly one element $x$ in $S$ such that by removing $x$ from both parts, $P$ and $P'$, the resulting partitions of $S\setminus \{x\}$ are the same. When $r=2$, we call this graph the \emph{Radon partition graph} of $S$.
\end{definition}

This definition implies that it is possible to reassign this special point $x$ from its subset in the partition $P$ to another subset of $P$ to obtain the partition $P^{\prime}$.   

Even though most known results around Tverberg's theorem deal with the case when $|S|=(d+1)(r-1)+1$, we will see that the structure of Tverberg graphs is richer when $|S|>(d+1)(r-1)+1$.

By ignoring all geometric information about the set of points $S$, it becomes natural to introduce an equivalent graph for the family of all non-empty partitions of a set $S$ into $r$ parts, which we denote by $G[S,r]$. Although some interesting properties of these graphs have been studied before (see for instance \cite{regnier1983quelques, almudevar1999estimation, konovalov2005partition}),  to our best of our knowledge, these graphs have not been defined previously.  Clearly, once the geometric information of $S$ is considered, only some of the vertices of $G[S,r]$ will correspond to Tverberg partitions of $S$. In other words, $G_T[S,r]$ is a subgraph of $G[S,r]$. We will show that these two graphs become increasingly similar in the sense that as $|S|$ becomes larger, almost every vertex of $G[S,r]$ is a vertex of $G_T[S,r]$, and almost all vertices in $G_T[S,r]$ have the same neighborhood as in $G[S,r]$. This reinforces recent findings related to Tverberg's theorem with tolerance, which shows that for large $S$, almost all partitions of $S$ into $r$ parts are Tverberg partitions \cite{Soberon:2018gn}.

\subsection{Main Results}

In our first main result, we give bounds for the number of elements in $S$ that guarantee the connectedness of the Tverberg $r$-partition graph of $S$. 

\begin{theorem}\label{generalconnected}
 Let $r$ and $d$ be positive integers and $S \subset \rr^d$ a set of at least $3(\tv(d,r))-1$ points. 
 Then $G_T[S,r]$ is connected.
 \end{theorem}

 It is not always the case that $3\tv(d,r)-1$ points are needed to guarantee connectivity of $G_T[S,r]$. For instance, we show in Theorem \ref{radon is alway connected} that a Radon partition graph is always connected.
 
 On the other hand, having cardinality $|S|= \tv(d,r)$ ensures that the graph $G_T[S,r]$ has a non-empty set of vertices, but it could have no edges.  A classic construction that exemplifies Tverberg's theorem is to take $(d+1)(r-1)+1$ points in $\rr^d$ in the following way.

 \begin{example}\label{eg:1}
 Consider $d+1$ clusters of $r-1$ points each, centered around the vertices of a fixed simplex $\Delta$. Then take an additional point $p$ in the barycenter of $\Delta$. Perturb each point slightly so that the family is in general position.  Observe that the only Tverberg partitions into $r$ parts have one part in the set $\{p\}$, while the remaining  $r-1$ parts each have one point of each cluster.  It is clear that it is not possible to go from one of these partitions to another by changing a single point, so the associated Tverberg partition graph has no edges.  For more examples of sets of $(r-1)(d+1)+1$ points with where the parts of each Tverberg partition have fixed sizes, see \cite{White2017}.
\end{example}

%We expect that the threshold for the graph $G_T[S,r]$ to be connected will strongly depend on...  

In any configuration of points where a set $S$ has more than $\tv(d,r)$ points, we have the following bounds for maximum and minimum degree in the Tverberg $r$-partition graph of $S$.

 \begin{theorem}\label{degree}
Let $S$ be a set of $n> \tv(d,r)$ points in $\rr^d$. Then for every vertex $v$ in $V(G_T[S,r])$, we have
\[
\Big(n+1 - \tv(d,r) \Big)(r-1)\leq \operatorname{deg}(v) \leq n(r-1).
\]  
\end{theorem}

Note that Example \ref{eg:1} also shows that the condition $n>\tv(d,r)$ is crucial, as the minimum degree may be zero if $n=\tv(d,r)$.  We conjecture that the lower bound given in Theorem \ref{degree} is always optimal when the set of points are in strong general position.  

\begin{conjecture}\label{minimum degree}
For any set of $n > \tv(d,r)$ points in $\rr^d$ in strong general position, the minimum degree of $G_T[S,r]$, denoted by $\delta(G_T[S,r])$, is exactly $(n+1 - \tv(d,r) )(r-1)$.  
\end{conjecture}

In the plane, we show (Proposition \ref{convex cases}) that this conjecture holds when the points in $S$ are in strong general convex position, and that when $r=2$, the strong general position condition can be dropped (Proposition \ref{minimum radon}), but we have an example that shows that the strong general position condition is essential (see Figure 3).  

%In Proposition \ref{minimum radon}, we prove that Conjecture \ref{minimum degree} is valid for the case when $r=2$ and $d\geq 2$. 

%In Proposition \ref{convex cases} we prove that this Conjecture holds for the case when $d=2$ and $r\geq2$, and the points in $S$ are in strong general convex position.
\medskip 

Finally, we prove that $G_T[S,r]$ asymptotically converges to $G[S,r]$ in the sense that the probability of a partition in $G[S,r]$ being also a Tverberg partition, and thus an element in $G_T[S,r]$, converges to one as the cardinality of $S$ increases. 

\begin{theorem} \label{thm:prob}
    Let $r$ and $d$ be fixed positive integers, and $S$ a set of $n$ points in $\rr^d$.  We can form a partition $P$ of $S$ into at most $r$ parts by assigning to each point of $S$ an integer from $1$ to $r$ uniformly and independently. Then the probability that $P$ is a Tverberg partition of $r$ parts with degree $n(r-1)$ in $G_T[S,r]$ is at least
    \[
1- \exp\left(-O(n)\right).
    \]
\end{theorem}

As we mentioned before, this result reinforces a recent finding related to Tverberg's theorem with tolerance, which is a generalization of Tverberg's theorem that introduces a new parameter $t$ called \emph{tolerance}.  Tverberg's theorem with tolerance states that there is an integer $N = N(d, t,r)$ such that any point set $S \subset \rr^d$ with at least $N$ points can be partitioned into $r$ disjoint sets $S_1,\dots, S_r$ such that $\bigcap_{j=1}^r \conv(S_j \setminus Q ) \neq \emptyset$ for any set $Q\subset S$ with at most $t$ points.  We observe that any partition with $t\geq 0$ tolerance is a Tverberg partition.  If a partition has tolerance for a $t\geq 1$, it is possible to relocate any element of $S$ to another subset of the partition and still have a Tverberg partition (since even removing the point still gives us a Tverberg partition). Thus, such a partition is a vertex in $G_T[S,r]$ of maximal degree. The current best asymptotic bound for $N(d,t,r)$ is $N(d,t,r) = rt + \tilde{O}(r^2\sqrt{td}+ r^3d)$, where $\tilde{O}$ hides polylogarithmic factors in $r,t,d$ \cite{Soberon:2018gn}.

\medskip
The rest of the paper is structured as follows. In Section \ref{sec:partitiongraphs} we describe the properties of $G[S,r]$, the general partition graphs on a set of $r$ elements.  These graphs are useful as a reference point when we ignore all geometric conditions on the set.  In Section \ref{sec:tverberggraphs} we prove our results for Tverberg partition graphs.  In Section \ref{radon and bounds} we prove stronger results for Radon partition graphs.  Finally, we present conclusions and future directions of research in Section \ref{sec:remarks}.

\section{Partition graphs}\label{sec:partitiongraphs}
Given a finite set $S$, a partition of $S$ is a collection of pairwise disjoint subsets (parts) whose union is $S$. Given two partitions $P$ and $P^\prime$ of a set $S$, we define the \emph{partition-distance} from $P$ to $P^\prime$ as the minimum number of elements we need to change from one part to another in $P$ to get $P^\prime$, and denote it as $D(P, P^\prime)$. The partition distance is also equal to the minimum number of elements of $S$ that we need to remove so that the partitions $P$ and $P^\prime$ restricted to the remaining points are equal. This definition  was first stated in 1965 by S.\ R\'{e}gnier \cite{regnier1983quelques}.
The distance between partitions has been widely studied in recent years, generating interesting results in areas such as bioinformatics \cite{almudevar1999estimation, konovalov2005partition} and data mining \cite{gusfield2002partition}.\\
In this paper we will work with \emph{$r$-partitions} of $S$, that is, partitions of $S$ into exactly $r$ nonempty parts. 

\begin{definition}
Given a set $S$ with $n$ elements and a positive integer $r$, the \emph{$r$-partition graph} of $S$, which we denote by $G[S,r]$, has as vertex set all $r$-partitions of $S$, and an edge between two of its partitions $P_1,P_2$ if and only if their partition distance $D(P_1, P_2)=1$. We will denote by $V(G[S,r])$ the set of vertices and by $E(G[S,r])$ the set of edges of $G[S,r]$.
\end{definition}

When $r=2$ and $n$ is even, note that all partitions may be divided into partitions with an odd number of elements or partitions with an even number of elements; moreover the distance between such partitions is one. Therefore for this case, this graphs are bipartite (see example in Figure \ref{prisms}).    

\begin{figure}[h]
\centering
\tikzset{every picture/.style={line width=0.75pt}} %set default line width to 0.75pt        

\begin{tikzpicture}[x=0.75pt,y=0.75pt,yscale=-1,xscale=1]
%uncomment if require: \path (0,154); %set diagram left start at 0, and has height of 154

%Shape: Ellipse [id:dp6040815786041172] 
\draw  [color={rgb, 255:red, 0; green, 0; blue, 0 }  ,draw opacity=1 ][fill={rgb, 255:red, 0; green, 0; blue, 0 }  ,fill opacity=1 ][line width=1.5]  (201.95,59.68) .. controls (201.94,58.33) and (203.01,57.24) .. (204.33,57.23) .. controls (205.66,57.23) and (206.73,58.32) .. (206.74,59.67) .. controls (206.74,61.02) and (205.67,62.12) .. (204.35,62.12) .. controls (203.03,62.13) and (201.95,61.04) .. (201.95,59.68) -- cycle ;
%Shape: Ellipse [id:dp37382023283996313] 
\draw  [color={rgb, 255:red, 0; green, 0; blue, 0 }  ,draw opacity=1 ][fill={rgb, 255:red, 0; green, 0; blue, 0 }  ,fill opacity=1 ][line width=1.5]  (202.02,117.48) .. controls (202.02,116.13) and (203.09,115.04) .. (204.41,115.03) .. controls (205.73,115.03) and (206.81,116.12) .. (206.81,117.47) .. controls (206.82,118.82) and (205.75,119.92) .. (204.42,119.92) .. controls (203.1,119.93) and (202.03,118.84) .. (202.02,117.48) -- cycle ;
%Shape: Ellipse [id:dp36344078855186623] 
\draw  [color={rgb, 255:red, 0; green, 0; blue, 0 }  ,draw opacity=1 ][fill={rgb, 255:red, 0; green, 0; blue, 0 }  ,fill opacity=1 ][line width=1.5]  (281.4,70.34) .. controls (281.39,68.99) and (282.46,67.89) .. (283.78,67.88) .. controls (285.11,67.88) and (286.18,68.97) .. (286.19,70.32) .. controls (286.19,71.67) and (285.12,72.77) .. (283.8,72.77) .. controls (282.48,72.78) and (281.4,71.69) .. (281.4,70.34) -- cycle ;
%Shape: Ellipse [id:dp39272650975387347] 
\draw  [color={rgb, 255:red, 0; green, 0; blue, 0 }  ,draw opacity=1 ][fill={rgb, 255:red, 0; green, 0; blue, 0 }  ,fill opacity=1 ][line width=1.5]  (202.12,88.31) .. controls (202.11,86.96) and (203.18,85.86) .. (204.5,85.86) .. controls (205.83,85.85) and (206.9,86.94) .. (206.91,88.3) .. controls (206.91,89.65) and (205.84,90.74) .. (204.52,90.75) .. controls (203.19,90.75) and (202.12,89.66) .. (202.12,88.31) -- cycle ;
%Shape: Ellipse [id:dp3264202714472999] 
\draw  [color={rgb, 255:red, 0; green, 0; blue, 0 }  ,draw opacity=1 ][fill={rgb, 255:red, 0; green, 0; blue, 0 }  ,fill opacity=1 ][line width=1.5]  (282.22,100.54) .. controls (282.21,99.19) and (283.28,98.09) .. (284.6,98.09) .. controls (285.93,98.09) and (287,99.18) .. (287.01,100.53) .. controls (287.01,101.88) and (285.94,102.98) .. (284.62,102.98) .. controls (283.29,102.98) and (282.22,101.89) .. (282.22,100.54) -- cycle ;
%Shape: Ellipse [id:dp6503622016894512] 
\draw  [color={rgb, 255:red, 0; green, 0; blue, 0 }  ,draw opacity=1 ][fill={rgb, 255:red, 0; green, 0; blue, 0 }  ,fill opacity=1 ][line width=1.5]  (201.98,29.4) .. controls (201.97,28.05) and (203.04,26.95) .. (204.37,26.94) .. controls (205.69,26.94) and (206.76,28.03) .. (206.77,29.38) .. controls (206.77,30.73) and (205.7,31.83) .. (204.38,31.83) .. controls (203.06,31.84) and (201.98,30.75) .. (201.98,29.4) -- cycle ;
%Shape: Ellipse [id:dp19643351792075991] 
\draw  [color={rgb, 255:red, 0; green, 0; blue, 0 }  ,draw opacity=1 ][fill={rgb, 255:red, 0; green, 0; blue, 0 }  ,fill opacity=1 ][line width=1.5]  (281.58,40.33) .. controls (281.57,38.98) and (282.64,37.88) .. (283.97,37.88) .. controls (285.29,37.87) and (286.36,38.97) .. (286.37,40.32) .. controls (286.37,41.67) and (285.3,42.76) .. (283.98,42.77) .. controls (282.66,42.77) and (281.58,41.68) .. (281.58,40.33) -- cycle ;
%Straight Lines [id:da7302350140549274] 
\draw    (204.37,29.39) -- (283.97,40.32) ;
%Straight Lines [id:da23793217905918818] 
\draw    (204.37,29.39) -- (283.79,70.33) ;
%Straight Lines [id:da5901919661765995] 
\draw    (204.37,29.39) -- (284.61,100.54) ;
%Straight Lines [id:da6972247077815785] 
\draw    (204.34,59.68) -- (283.97,40.32) ;
%Straight Lines [id:da5003292330089637] 
\draw    (204.34,59.68) -- (283.79,70.33) ;
%Straight Lines [id:da24722244560830564] 
\draw    (204.34,59.68) -- (284.61,100.54) ;
%Straight Lines [id:da5934941691876052] 
\draw    (204.51,88.3) -- (284.61,100.54) ;
%Straight Lines [id:da44561102774076367] 
\draw    (204.51,88.3) -- (283.79,70.33) ;
%Straight Lines [id:da6524566686702455] 
\draw    (204.42,117.48) -- (284.61,100.54) ;
%Straight Lines [id:da2847479231479155] 
\draw    (204.51,88.3) -- (283.97,40.32) ;
%Straight Lines [id:da5259237821621832] 
\draw    (204.42,117.48) -- (283.79,70.33) ;
%Straight Lines [id:da208173308236292] 
\draw    (204.42,117.48) -- (283.93,39.74) ;

% Text Node
\draw (146.1,23.97) node [anchor=north west][inner sep=0.75pt]   [align=left] {{\tiny \{1\},\{2,3,4\}}};
% Text Node
\draw (146.1,53.3) node [anchor=north west][inner sep=0.75pt]   [align=left] {{\tiny \{2\},\{1,3,4\}}};
% Text Node
\draw (146.43,82.97) node [anchor=north west][inner sep=0.75pt]   [align=left] {{\tiny \{3\},\{1,2,4\}}};
% Text Node
\draw (146.77,112.3) node [anchor=north west][inner sep=0.75pt]   [align=left] {{\tiny \{4\},\{1,2,3\}}};
% Text Node
\draw (292.43,35.17) node [anchor=north west][inner sep=0.75pt]   [align=left] {{\tiny \{1,2\},\{3,4\}}};
% Text Node
\draw (292.77,64.83) node [anchor=north west][inner sep=0.75pt]   [align=left] {{\tiny \{1,3\},\{2,4\}}};
% Text Node
\draw (292.77,94.83) node [anchor=north west][inner sep=0.75pt]   [align=left] {{\tiny \{1,4\},\{2,3\}}};

\end{tikzpicture}
\caption{The $2$-partition graph of the set $\{1,2,3,4\}$.}\label{prisms}
\end{figure}
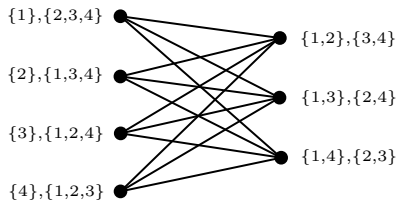

Given positive numbers $n$ and $r$, let $S(n,r)$ denote the number of Stirling numbers of the second kind (see \cite[p.~76]{riordan2012introduction}). By using simple counting arguments, the following proposition analyses the basic properties of $G[S,r]$.  

\begin{proposition}
Given a set $S$ with $n$ elements and a positive integer $r$,  $G[S, r]$ is a connected graph with $|V(G[S, r])=S(n,r)$, and for every vertex $v\in G[S, r]$, its minimum degree is $(n-r+1)(r-1)$ and the maximum degree is $n(r-1)$. Furthermore  if $r\leq n/2$,  Then \[\vert E(G[S,r])\vert= \frac{1}{2}\sum\limits_{k=0}^{r-1} {n\choose k}S_2(n-k,r-k)(n-k)(r-1),\] where $S_2(x,y)$ denotes the $2$-associated Stirling number of the second kind.
\end{proposition}

\begin{proof}

Since the Stirling numbers of the second kind count the number of different ways of dividing a set with $n$ elements into $r$ non-empty subsets (parts), the graph $G[S, r]$ will have clearly exactly $S(n,r)$ vertices. Given a partition $P= P_1,\dots ,P_r $ of $S$, we note that if an element lies in a part with more than one element, we can move it to any of the remaining parts. Furthermore the degree of $P$ in $G$ is equal to $\sum_{\vert P_i \vert \neq 1} \vert P_i \vert(r-1)$, which implies that the minimum degree of $G[S,r]$ is $(n-r+1)(r-1)$ and the maximum degree is $n(r-1)$, and the $r$-partition graph is connected.

Let $P$ be a partition of $S$ with exactly $k$ parts of cardinality one. Then $\deg(P)=(n-k)(r-1)$ because the $(n-k)$ elements in the parts with cardinality greater than one can be moved to any of the remaining $r-1$ parts. There are exactly ${n \choose k} S_2(n-k,r-k)$ possible partitions with this degree, where the $2$-associated Stirling numbers of the second kind $S_2(x,y)$ count the number of ways to partition a set of $x$ objects into $y$ subsets, with each subset containing at least $2$ elements. Then the sum of the degrees of all the vertices is \[\sum\limits_{k=0}^{r-1}{n \choose k} S_2(n-k,r-k)(n-k)(r-1).\]
Since $E=\frac{1}{2}\sum\limits_{P \in V} deg(P)$, we have the desired result. 
\end{proof}

Table \ref{nincrease} and \ref{rnincrease}, shows how rapidly the number of vertices and the number of edges increases as $r$ and $n$ increase, respectively. 
 
\begin{table}[h]
\centering
\begin{tabular}{|ccccccccc|}
\hline
\multicolumn{9}{|c|}{\begin{tabular}[c]{@{}c@{}}Number of vertices for the \\ $r$-partition graph of a set with $n$ elements\end{tabular}}                                                                                                                                    \\ \hline
\multicolumn{1}{|c|}{\textbf{r}\textbackslash{}\textbf{n}}        & \multicolumn{1}{c|}{\textbf{5}} & \multicolumn{1}{c|}{\textbf{6}} & \multicolumn{1}{c|}{\textbf{7}} & \multicolumn{1}{c|}{\textbf{8}} & \multicolumn{1}{c|}{\textbf{9}} & \multicolumn{1}{c|}{\textbf{10}} & \multicolumn{1}{c|}{\textbf{11}} & \textbf{12} \\ \hline
\multicolumn{1}{|c|}{\textbf{2}} & \multicolumn{1}{c|}{15}         & \multicolumn{1}{c|}{31}         & \multicolumn{1}{c|}{63}         & \multicolumn{1}{c|}{127}        & \multicolumn{1}{c|}{255}        & \multicolumn{1}{c|}{511}         & \multicolumn{1}{c|}{1023}        & 2047        \\ \hline
\multicolumn{1}{|c|}{\textbf{3}} & \multicolumn{1}{c|}{25}         & \multicolumn{1}{c|}{90}         & \multicolumn{1}{c|}{301}        & \multicolumn{1}{c|}{966}        & \multicolumn{1}{c|}{3025}       & \multicolumn{1}{c|}{9330}        & \multicolumn{1}{c|}{28501}       & 86526       \\ \hline
\multicolumn{1}{|c|}{\textbf{4}} & \multicolumn{1}{c|}{10}         & \multicolumn{1}{c|}{65}         & \multicolumn{1}{c|}{350}        & \multicolumn{1}{c|}{1701}       & \multicolumn{1}{c|}{7770}       & \multicolumn{1}{c|}{34105}       & \multicolumn{1}{c|}{145750}      & 611501      \\ \hline
\multicolumn{1}{|c|}{\textbf{5}} & \multicolumn{1}{c|}{1}          & \multicolumn{1}{c|}{15}         & \multicolumn{1}{c|}{140}        & \multicolumn{1}{c|}{1050}       & \multicolumn{1}{c|}{6951}       & \multicolumn{1}{c|}{42525}       & \multicolumn{1}{c|}{246730}      & 1379400     \\ \hline
\end{tabular}
\caption{Number of vertices in partition graphs with respect to $r$ and $n$. } 
\label{nincrease}
\end{table}

\begin{table}[h]
\centering
\begin{tabular}{|ccccccccc|}
\hline
\multicolumn{9}{|c|}{\begin{tabular}[c]{@{}c@{}}Number of edges for the \\ $r$-partition graph of a set with $n$ elements\end{tabular}}                              
%\multicolumn{9}{|c|}{\begin{tabular}[c]{@{}c@{}}Número de aristas de la gráfica de\\ r-particiones  de un conjunto con n elementos\end{tabular}}                    
\\ \hline
\multicolumn{1}{|c|}{\boldmath{$r\backslash n$}} & \multicolumn{1}{c|}{\textbf{5}} & \multicolumn{1}{c|}{\textbf{6}} & \multicolumn{1}{c|}{\textbf{7}} & \multicolumn{1}{c|}{\textbf{8}} & \multicolumn{1}{c|}{\textbf{9}} & \multicolumn{1}{c|}{\textbf{10}} & \multicolumn{1}{c|}{\textbf{11}} & \textbf{12} \\ \hline
\multicolumn{1}{|c|}{\textbf{2}}                  & \multicolumn{1}{c|}{35}         & \multicolumn{1}{c|}{90}         & \multicolumn{1}{c|}{217}        & \multicolumn{1}{c|}{504}        & \multicolumn{1}{c|}{1143}       & \multicolumn{1}{c|}{2550}        & \multicolumn{1}{c|}{5621}        & 12276       \\ \hline
\multicolumn{1}{|c|}{\textbf{3}}                  & \multicolumn{1}{c|}{90}         & \multicolumn{1}{c|}{450}        & \multicolumn{1}{c|}{1890}       & \multicolumn{1}{c|}{7224}       & \multicolumn{1}{c|}{26082}      & \multicolumn{1}{c|}{90750}       & \multicolumn{1}{c|}{307890}      & 1026036     \\ \hline
\multicolumn{1}{|c|}{\textbf{4}}                  & \multicolumn{1}{c|}{30}         & \multicolumn{1}{c|}{360}        & \multicolumn{1}{c|}{2730}       & \multicolumn{1}{c|}{16800}      & \multicolumn{1}{c|}{91854}      & \multicolumn{1}{c|}{466200}      & \multicolumn{1}{c|}{2250930}     & 10494000    \\ \hline
\multicolumn{1}{|c|}{\textbf{5}}                  & \multicolumn{1}{c|}{0}          & \multicolumn{1}{c|}{60}         & \multicolumn{1}{c|}{1050}       & \multicolumn{1}{c|}{11200}      & \multicolumn{1}{c|}{94500}      & \multicolumn{1}{c|}{695100}      & \multicolumn{1}{c|}{4677750}     & 29607600    \\ \hline
\end{tabular}
\caption{Number of edges in partition graphs with respect to $r$ and $n$. } \label{rnincrease}
\end{table}

Given two $r$-partitions $P$ and $P^\prime$, the distance $D(P,P^\prime)$ is exactly the minimum length over all paths between $P$ and $P^\prime$ in $G[S,r]$. Then the \emph{diameter} of the graph $G[S,r]$, denoted by $\diam(G[S,r])$, can be interpreted as the maximum of the distances over all possible $r$-partitions of $S$. In \cite{charon2006maximum}, it was proven that
\[
\diam(G[S,r])= \left\{
\begin{array}{c l}
 2n-2r &  \text{ if } n\leq 2r-2\\
 n-\lceil \frac{n}{r} \rceil & \text{ if }  n\geq 2r-1,
\end{array}
\right.
\]
where $n$ is the number of elements in $S$.

The clique number of a graph $G$, denoted by $\omega(G)$, is the number of vertices in its largest clique. The following result gives the clique number of the partition graph of a set on any number of parts.

\begin{proposition}\label{cliquepartitiongraph}
Let $S$ be a set with $n$ elements, and $G[S,r]$, $r\geq 1$,  be the $r$-partition graph for an integer $n>r$. Then $\omega(G[S,r])=r$. 
\end{proposition}

\begin{proof}
Consider a partition $P$ of $S$.  Since $n> r$, there is at least one part with more than one element.  We can fix one of the elements $x$ in a part with more than one element.  If we move $x$ to any of the other parts, the resulting $r-1$ partitions, together with $P$, induce a complete subgraph of $G[S,r]$ with exactly $r$ vertices. Then $\omega(G[S,r])\geq r$. 

On the other hand, if $Q$ is a clique, we will prove that all the vertices of $Q$ differ in exactly one specific element.
If $Q$ is an edge, by definition we are done. So, we can suppose that $Q$ has $k\geq 3$ vertices, say $P_1,\dots ,P_k$. Since $P_1$ is adjacent to $P_2$, there exists an element $x$ in $S$ such that if we delete it we obtain that $P_1$ is equal to $P_2$ (restricted to $S\setminus \{x\}$). Analogously, there exists an element $y$ in $S$ such that $P_2=P_3$, restricted to $S\setminus \{y\}$. If these two elements were different, then $D(P_1,P_2)>1$, which is a contradiction. Therefore, $x=y$.  Since all the vertices of $Q$ differ only by the element $x$, we have $\omega(G[S,r])\leq r$.
\end{proof}

We believe the graph $G[S,r]$ is interesting on its own, for instance it is not difficult to observe that when $n$ and $r=n-1$ the complement of the partition graph,  $\Bar{G[n,n-1]}$ is isomorphic to the Kneser graph $K(n,2)$ (see \cite{Godsil2001}) and when $|S|=n$ is even and $r=2$, $G[S,r]$ is bipartite.      

\section{Tverberg Partition Graphs}\label{sec:tverberggraphs}

In this section, we prove our results on Tverberg partition graphs.  Recall that given a finite set $S \subset \rr^d$, we denote by $G_T[S,r]$ its Tverberg partition graph for partitions into $r$ parts.  Each Tverberg partition with $r$ parts represents a vertex, and two Tverberg partitions $P$ and $P'$ form an edge if and only if $D(P,P')=1$.

To establish our first results, we will use the following theorem by Attila P\'{o}r \cite{attilaphd}, which generalizes Kirchberger's theorem.

\begin{theorem}[Attila P\'{o}r \cite{attilaphd} 1997]\label{Por}
Let $P_1,P_2,\dots,P_r$ be disjoint sets in $\rr^d$. Then $\bigcap_{j=1}^r \conv(P_j)\neq \emptyset$ if and only if for 
$j=1,\dots,r, $ there exists $P_j'\subset P_j$ such that 
$\bigcap_{j=1}^r \conv(P_j' )\neq \emptyset$ and $\left|\bigcup_{i=1}^r P_i'\right|\leq (d+1)(r-1)+1.$
\end{theorem}

%\begin{lemma}
%Let $S$ be a set of $n\geq \tv(d,m)$ points in $\rr^d$ and $G$ the Tververg partition graphs associated to $S$. Then, for every vertex $v$ in $V(G)$ we have
%$$(m-1)(n-\tv(d,m))\leq deg(v) \leq (m-1)n.$$
%\end{lemma}

%\begin{proof}
%Let $v\in V(G)$ an arbitrary vertex, i.e., $v$ is a Tverberg partition $v=P_1,...,P_m$ of $S$. By Pór's theorem there is a subset $S^{\prime}\subset S$ with exactly $\tv(d,n)$ points and a partition $v^{\prime}= P_1^{\prime},..,P_{m}^{\prime}$  with $P_i^{\prime}\subset P_i$ for $1\leq i \leq m$  such that $\cap_{i=1}^m \conv(P_i^\prime)\neq \emptyset$. Consider all neighbors of $v$, this is, all partitions that are at distance one form $v$. There exist exactly $(m-1)(n- \tv(d,m))$ such partitions. Then, for everyone of these partitions we have that every $P_i^\prime$, $1\leq i \leq m$ is contained in a different part, then these are also Tverberg partitions and we obtain the left hand side of the inequality. The right hand side of the equality follows simply from counting the maximum number of $m$-partitions that differ by exactly one element from $v$. 
%\end{proof}

% $P=((0, 0), (13, 27), (10, 33), (10, 8), (6, 39), (53, 46), (0, 47), (24, 31))$

% \begin{center}
% ¿Qué hay de la conexidad?\\
% ¿Máximo y mínimo grado?\\
% ¿Diametro?\\
% etc.    
% \end{center}

\subsection{Connectivity}

%\begin{proposition}\label{generalconnected}
 %Let $r,d$ be positive integers and $S \subset \rr^d$ be a set of at least $3 \tv(d,r)-1$ points.  Then, the $(d,r)$-Tverberg partition graph of $S$ is connected.
 %\end{proposition}

 Next we will show Theorem \ref{generalconnected}, which claims that $G_T[S,r]$ is a connected graph if the set $S$ of points contains at least $3\tv(d,r)-1$.

 \begin{proof}[Proof of Theorem \ref{generalconnected}]
Consider two different Tverberg partitions $P= P_1,\dots,P_r $ and $Q=Q_1,\dots,Q_r$ of $S$. By Theorem \ref{Por}, there exist subsets $S_P$ and $S_Q$ of $S$ with at most $\tv (d,r)$ points each, and such that $P'=P_1^\prime,\dots ,P_r^\prime$ is a Tverberg $r$-partition of $S_P$ with $P_j^\prime \subset P_j$ for $1\leq j \leq r$ and $Q'=Q_1^\prime,\dots ,Q_r^\prime$ is a Tverberg $r$-partition of $S_Q$ with $Q_j^\prime \subset Q_j$ for $1\leq j \leq r$.
We will construct a sequence of $r$-partitions of $S$, $P_{R_0},P_{R_1}\dots P_{R_k}$ that form a path between $P=P_{R_0}$ and $Q=P_{R_k}$ in $G_T[S,r]$.

Suppose that $S_P\cap S_Q = \emptyset$.  Let $x\in S\setminus S_P$ and assume that $x$ is say in partition $P_i$ in $P$ and $x$ is in partition $Q_j$ in $Q$. Then, starting with partition $P=P_{R_0}$, we begin by shifting only $x$ from partition $P_i$ to  partition $Q_j$. Then we obtain a new Tverberg partition $P_{R_1}$ that differs from $P$ in only one element, and it is closest to $Q$. Recursively we will move the points of $S\setminus S_P$ in $P_{R_t}$ one by one  from their corresponding part in $P$ to their corresponding part in $Q$. 
In each movement, the points of $S_P$ remain untouched, so each partition generated is a Tverberg partition. At the end of this procedure, all the points of $S_P$ and $S_Q$ remain untouched. Next we can similarly move the points of the set $S_P$ to the corresponding part until we get the partition $Q$.

% Suppose that $S_P\cap S_Q = \emptyset$.  Then, for each point $p \in S_P$ we can change which part in the partition it belongs to, so that, until we induce the partitions $Q'_1, \dots, Q'_r$ in $S_Q$.  We do this process for one point at a time.  Since we never changed an element of $S_P$, at every step the partition was a Tverberg partition.  Now, we can change all elements of $S\setminus S_Q$ to match the partition $Q_1, \dots, Q_r$.  Since we never changed an element of $S_Q$, at every step the partition was a Tverberg partition.

If $S_P\cap S_Q \neq \emptyset$, notice that the set $S_R = S\setminus (S_P \cup S_Q)$ has at least $\tv(d,r)$ points. Then there exists a Tverberg $r$-partition $R$ of $S$ such that $R$ restricted to the points of $S_R$ remains a Tverberg partition. Proceeding as above, we can find a path of Radon partitions going from $P$ to $R$ and another path from $R$ to $Q$. 
% Now, starting with the partition $P=P_1,\dots, P_r$, we first change for each element in $S'$ the part it belongs to until we match $\mathcal{P}$ in $S'.$  We do this one point at a time.  Since we did not modify the induced partition on $S_P$, at every step we had a Tverberg partition.  Then, we change the part of each point in $S_Q$ so that we match the partition $Q'_1, \dots, Q'_r$.  Since we never changed the partition induced by $S'$, at every point the partition was a Tverberg partition.  Finally, we change the part of each point in $S\setminus S_u$ until we match the partition $Q_1,\dots,Q_r$.  We do this one point at a time.  Since we never modified the partition induced by $S_Q$, at every step we have a Tverberg partition.
\end{proof}

We believe that the bound of $3\tv(d,r)-1$ in Theorem \ref{generalconnected} is not optimal. In Theorem \ref{radon is alway connected} in Section \ref{radon and bounds}, we will show that in the case of $r=2$, the Radon partition graph is always connected. We also analyzed   
Tverberg $3$-partition graphs induced by at least one representative in every order type for sets with $n\in\{7,8,9,10\}$ points. We generated over 2000 Tverberg $4$-partition graphs that belong to different order types chosen randomly and uniformly, observing that in the plane all Tverberg graphs become connected as soon as we have $\tv(2,r)+1$ points.  We conjecture the following. 

\begin{conjecture}
There exists a constant $k(d)$ that depends only on the dimension $d$ such that the Tverberg $r$-partition graph $G_T[S,r]$ of a set of points $S\in R^d$ is connected if $\vert S \vert \geq \tv(d,r)+k(d)$, with $k(1)=1$.
\end{conjecture}

\subsection{Maximum and minimum degree}\label{MaxyMin}

%In this subsection will investigate the maximum and minimum degree of Tverberg partition graphs.

%\begin{theorem}\label{degree}
%Let $S$ be a set of $n> \tv(d,r)$ points in $\rr^d$ and $G$ be the $(d,r)$-Tverberg partition graph for $S$. Then, for every vertex $v$ in $V(G)$ we have
%$$\Big(n+1 - \tv(d,r) \Big)(r-1)\leq \operatorname{deg}(v) \leq n(r-1).$$  
%\end{theorem}

%\begin{lemma}
%Let $G$ be the Tverberg graph of a set of $n$ points in $\rr^d$, for Tverberg partitions into $r$ parts.  Then, the minimal degree of $G$ is at least
%\[
%\Big(n-(r-1)(d+1)\Big)(r-1) = \Big(n+1 - \tv(d,r) \Big)(r-1)
%\]
%\end{lemma}

To prove our results related to the minimum degree of Tverberg graphs, we use the following strong form of the Colorful Carath\'eodory theorem \cite{Barany1982}.

\begin{theorem}[B\'ar\'any 1982]\label{caratheodory}
    Let $d$ be a positive integer and $p$ and $q_{d+1}$ be two points in $\rr^d$.  If we have $d$ sets $X_1, \dots, X_d$ in $\rr^d$ such that $p \in \conv (X_i)$ for $i=1,\dots, d$, we can choose $d$ points $q_1 \in X_1, \dots, q_d \in X_d $ such that
    \[
p \in \conv \{q_1, q_2, \dots, q_{d+1}\}.
    \]
\end{theorem}

Next we will prove Theorem \ref{degree}, which asserts that the degree $\delta(v)$ of any vertex $v\in V(G_T[S,r])$ is bounded by $\Big(n+1 - \tv(d,r) \Big)(r-1)\leq \operatorname{deg}(v) \leq n(r-1)$ when $n>\tv(d,r)$.

\begin{proof}[Proof of Theorem \ref{degree}]
We can assume without loss of generality that the set of points $S$ is algebraically independent.  Let $P=P_1, P_2, \dots, P_r$ be a Tverberg partition of $S$.

By Theorem \ref{Por}, there exists a subset $S' \subset S$ of at most $\tv(d,r)$ points such that the sets $P'_i = P_i \cap S'$ for $i=1,\dots, r$ form a Tverberg partition.   Note that we can change any point in $S \setminus S'$ in a partition, say $P_i$, to any of the remaining partitions and still have a Tverberg partition.  This gives us $|S\setminus S'|\cdot (r-1) = (n-\tv(d,r))(r-1)$ edges in $G_T[S,r]$ incident in $P$.  We need $r-1$ more edges to achieve the lower bound.

We claim that there exists $x \in S'$ such that $S \setminus \{x\}$ is a Tverberg partition.  If this holds, and $x$ is in say partition  $P_i$ for some $i$, it means that we can change $x$ to any of the $r-1$ remaining partitions and still have a Tverberg partition.  

Since the points of $S$ are algebraically independent, we know that $\bigcap_{j=1}^r \conv(P'_j)$ is a single point $p$ and it is in the relative interior of each $\conv(P'_j)$.  We also know that $|P'_j| \le d+1$ for each $j=1,\dots, r$.

Let $y \in S\setminus S'$.  We assume without loss of generality that $y \in P_1$. 

Assume first that  $\vert P^\prime_1 \vert = d+1$.  If this is the case, we can use the Colorful Carath\'eodory Theorem \ref{caratheodory} with $d$ copies of $P'_1$ and one copy of $\{y\}$ (playing the role of $q_{d+1}$).  Since the $d$ copies contain $p$ in their convex hull, there is a choice of at most $d$ points of $P'_1$ such that the convex hull of those points and $y$ contains $p$.  In other words, there exists a point $x \in P'_1$ such that $p \in \conv((P'_1 \setminus \{x\}) \cup \{y\})$, as we wanted.

Then we may assume that we have $|P'_1| \le d$, so the dimension of $\conv (P'_1)$ is at most $d-1$.  Therefore, $\dim (\conv((P'_1 \cup \{y\})) = 1 + \dim(\conv(P'_1))$.  This means that
\[
\dim \Big(\conv(P'_1 \cup \{y\}) \cap \conv(P'_2) \cap \dots \cap \conv(P'_r) \Big) = 1.
\]
Let $\ell = \conv((P'_1 \cup \{y\}) \cap \conv(P'_2) \cap \dots \cap \conv(P'_r)$.  By the arguments above, $\ell$ is a segment, with one endpoint being $p$.  Let $q$ be the other endpoint.  Since $q$ is an endpoint, it means that $q$ is either on the boundary of a set of the form $\conv(P'_j)$ with $j>1$ or on the boundary of $\conv((P'_1 \cup \{y\})$.  In the first case, there exists $x \in P'_j$ such that $q \in \conv(P'_j \setminus {x})$.  In the second case, since $q \neq p$, there exists $x \in P'_j$ such that $q \in \conv(P'_1\setminus\{x\} \cup \{y\})$.  In either case, we know that $x \in S'$, and if we remove $x$ we still have a Tverberg partition, as we wanted. 
\medskip

The right-hand side of the equality follows simply from counting the maximum number of $r$-partitions that differ by exactly one element from $P$. \end{proof}

Next, we will relate the maximum degree of the Tverberg partition graph $G$ to Tverberg's theorem with tolerance.  Recall that $N(d,t,r)$ is the minimum number of points in $\rr^d$ to guarantee the existence of Tverberg partitions with tolerance $t$.  As mentioned in the introduction, the current best asymptotic bound for $N(d,t,r)$ is $N(d,t,r) = rt + \tilde{O}(r^2\sqrt{td}+ r^3d)$, where $\tilde{O}$ hides polylogarithmic factors in $r,t,d$ \cite{Soberon:2018gn}.  For small values of $t,d,r$, we have the bound $N(d,t,r)\le (t+1)(r-1)(d+1) + 1$ in general, and the bound $N(d,t,r) \leq 2^{d-1}(r(t+2)-1)$, which is better for low dimensions \cite{Soberon:2012er, mulzer2014algorithms}.

This means that as long as $|S| \ge N(d,1,r) \ge \min \{ 2(r-1)(d+1)+1, 2^{d-1}(3t-1), rt + \tilde{O}(r^2\sqrt{td}+ r^3d)\}$, the Tverberg graph for $S$ will have at least one vertex of maximum degree.  We prove a stronger result, showing that as $|S|$ increases, almost all the vertices of the Tverberg graph will have maximum degree.  First, we state a precise version of Theorem \ref{thm:prob}.

%can be foundIn \cite{mulzer2014algorithms} Mulzer and Stein gave the bound $N(d,t,r) \leq 2^{d-1}(r(t+2)-1)$. Then, if $S$ have at least $2^{d-1}(3r-1)$ points there exists a vertex $v$ such that $deg(v)=n(r-1)$.
\medskip
\noindent\textbf{Theorem \ref{thm:prob}.} 
\emph{Let $n,r,d$ be positive integers and $S$ be a set of $n$ points in $\rr^d$.  We assign to each point of $S$ a label from $1,\dots, r$, uniformly and independently.  This forms a random partition of $S$.  The probability that such a partition is a Tverberg partition into $r$ parts with degree $n(r-1)$ in $G_T[S,r]$ is at least
    \[
1- r^{(r-1)(d+1)}n^{(r-1)(d+1)}\exp\left(\frac{-2(n-r)^2}{nr^2}\right).
    \]
}
\begin{proof}
    To prove this result, we first use Sarkaria's transformation \cite{Sarkaria:1992vt}.  This is a standard technique to reduce Tverberg-type problems to variations of the colorful Carath\'eodory theorem.  The version we use is a simplification by B\'ar\'any and Onn \cite{Barany:1995tg}.

    We introduce $r$ vectors $v_1, v_2, \dots, v_r$, which are the vertices of a regular simplex in $\rr^{r-1}$ centered at the origin.  Then for each $a \in \rr^d$, we consider the set
    \[
    X_a = \{(a,1) \otimes v_j \in \rr^{(d+1)(r-1)} \colon j =1,\dots, r\},
    \]
    where $\otimes$ stands for the tensor product.  Given a set $a_1, \dots, a_n$ of points in $\rr^d$, picking a representative of each set $X_{a_1}, \dots, X_{a_n}$ is equivalent to making a partition of $a_1, \dots, a_n$ into $r$ parts.  Moreover, this partition is a Tverberg partition if and only if the convex hull of the choice made in $\rr^{(r-1)(d+1)}$ contains the origin (see, e.g., \cite{Barany:2018fy}).

    Given $a_1, \dots, a_n$ points in $\rr^d$, we pick a point at random from each of $X_{a_1}, \dots, X_{a_n}$, uniformly and independently.  Let $N$ be the resulting set.  Let $H$ be a half-space in $\rr^{(r-1)(d+1)}$ whose boundary hyperplane contains the origin.  Let us bound the probability that $|H\cap N| \le 1$.

    Let $y_i$ be our choice from $X_{a_i}$.  Since $\conv(X_{a_i})$ contains the origin in $\rr^{(r-1)(d+1)}$, we have $H \cap X_{a_i} \neq \emptyset$. Therefore, $\mathbb{P} (y_i \in H) \ge 1/r$.  Using the linearity of expectation, 
    \[
\mathbb{E}[|H \cap N|] = \sum_{i=1}^n \mathbb{P} (y_i \in H) \ge n/r.
    \]
    Since this expectation is a sum of $n$ independent random variables that take values in $[0,1]$, we can apply Hoeffding's inequality to obtain that for every $0\le \lambda \le n/r$, we have
    \[
    \mathbb{P}(|H \cap N| \le n/r - \lambda) \le \mathbb{P}(|H \cap N| \le \mathbb{E}[|H \cap N|] - \lambda)\le \exp(-2\lambda^2/n).
    \]
    Let us choose $\lambda = (n/r)-1$, so we have
    \[
    \mathbb{P}(|H \cap N| \le 1) \le \exp(-2(n-r)^2/r^2n).
    \]
    Now, consider the set of points $X=\bigcup\{X_{a_i} \colon i =1,\dots, n\}\subset \rr^{(r-1)(d+1)}$. For each point $x \in X$, consider the hyperplane $H_x = \{y \colon \langle x,y\rangle = 0\}$, where $\langle \cdot, \cdot \rangle$ denotes the standard dot product.  This hyperplane arrangement has at most $|X|^{(r-1)(d+1)}$ cells, which means (by duality again) that there are at most $|X|^{(r-1)(d+1)}$ possible subsets of $X$ that a half-space $H$ as above can contain.  Therefore, a union bound shows that the probability that there exists a half-space $H$ for which $|H \cap N| \le 1$ is bounded above by 
    \[
  |X|^{(r-1)(d+1)}\exp\left(\frac{-2(n-r)^2}{nr^2}\right).
    \]
    Alternatively, if $|H \cap N| \ge 2$ for every possible $H$, removing any point of $N$ still gives us a set that contains the origin in $\rr^{(r-1)(d+1)}$.  This means that removing any of the points $a_1, \dots, a_n$ from the original set in $\rr^d$ still gives us a Tverberg partition.  We finish the proof by observing that $|X| = nr$.
\end{proof}

This result shows that as $n$ increases, almost all vertices of the Tverberg graph reach maximum degree very quickly, as the exponential decay shrinks much faster than $n^{(r-1)(d+1)}$.  Moreover, the Tverberg graph and the partition graph are extremely similar.

\subsection{Cliques}

As in the case of Proposition \ref{cliquepartitiongraph} on cliques of partition graphs, the following proposition shows that Tverberg partition graphs behave in a very similar way.

\begin{proposition}
Let $G_T[S,r]$ be the Tverberg $r$-partition graph of a set $S$ with $n>\tv(d,r)$ points in $\rr^d$.  Then the clique number $\omega(G_T[S,r])$ is $r$.
\end{proposition}

\begin{proof}
Since the Tverberg $r$-partition graph $G_T[S,r]$ is a subgraph of the partition graph $G[S,r]$, and by Proposition \ref{cliquepartitiongraph}, $\omega(G[S,r])=r$, then $\omega(G_T[S,r])\leq r$.
 Let $T_1,T_2,\dots,T_k$ be a family of Tverberg $r$-partitions that induces a clique in $G_T[S,r]$. Since $D(T_1,T_2)=D(T_2,T_3)=1$, there exists a point $x\in S$ and a part $T_{1x}$ of $T_1$ such that if we move $x$ to $T_{1x}$ we obtain $T_2$. Analogously, there exists a point $y\in S$ and a part $T_{2y}$ of $T_2$ such that if we move $y$ to $T_{2y}$ we obtain $T_3$.
 If $x\neq y$, then $D(T_1,T_3)>1$, but this is a contradiction, so $x=y$. Inductively, we have that the point at which all these partitions differ is exactly $x$, and this implies that $\omega(G_T[S,r])\leq r$. 

%By Theorem \ref{degree} we know $\delta(v)\geq (n+1-\tv(S,d))r-1>(r-1)$ since $n>\tv(S,d)$. 
Now let $P=P_{1}, P_{2}, \dots, P_{r}$ be a Tverberg partition of $S$. By P\'or's theorem, there exists a subset $S' \subset S$ of at most $\tv(d,r)$ points such that the sets $P'_{i} = P_{i} \cap S'$ for $i=1,\dots, r$ form a Tverberg partition. Consider a point $x\in S\setminus S^\prime$. Without loss of generality, suppose that $x\in P_{1}$. We can move $x$ to every part $P_{i}$ for $2\leq i \leq r$ and the resulting partitions $P_2,\dots ,P_r$ remain Tverberg. Moreover, every pair of the partitions $P_1,P_2,\dots,P_r$ are at distance one in $G_T[S,r]$, so they induce a complete subgraph of order $r$. 
Then $\omega(G_T[S,r])\geq r$.
\end{proof}

\section{Radon partition graphs and improved bounds}\label{radon and bounds}

%\textcolor{blue}{Let us remember that Tverberg's theorem is the generalization of Radon's theorem, therefore when $r=2$ a $2$-Tverberg partition is also called a Radon partition. In this section we are interested in obtaining better bounds for some of the structural properties that we studied in the previous section in the particular case of Radon graphs.}

In this section we are interested in obtaining better bounds on the same structural properties that we studied in the previous section by fixing the dimension, reducing the number of parts in our partitions, or by imposing additional conditions on the set of points.  

We begin by proving a version of Theorem \ref{generalconnected} when $r=2$; that is, when we have Radon partitions. In this case, we refer to their associated graphs as \emph{Radon partition graphs}. As we may observe, the following result does not even need the points to be in general position. 

\begin{theorem}\label{radon is alway connected}
The Radon partition graph $G_T[S,2]$ of any set $S$ of points $\rr^d$ is connected. 
\end{theorem}

\begin{proof}
If the set $S$ has at most one Radon partition, then $G_T[S,2]$ is the null graph or consists of a single vertex; in either case $G_T[S,2]$ is connected. Let $n$ be the number of points of $S$.

Suppose that there are at least two Radon partitions $P=P_1, P_2$ and $Q=Q_1, Q_2$ of $S$. We will construct a sequence of Radon partitions starting in $P$ and finishing in $Q$ such that two adjacent partitions differ by only one element.  
Since $P$ and $Q$ are Radon partitions, there exist $p=(\alpha_1,\alpha_2,\dots ,\alpha_{n})$ and $q=(\beta_1,\beta_2,\dots ,\beta_{n})$ points in $\rr^{n}$ with entries not all $0$ such that $\sum_{i=1}^{n}\alpha_i=0$, $\sum_{i=1}^{n}\beta_i=0$ and $\sum_{i=1}^{n}\alpha_i p_i=0$, $\sum_{i=1}^{n}\beta_i p_i=0$. Furthermore, we can assume that the points of $S$ that lie in the parts $P_1$ and $Q_1$ all have positive coefficients, while the points of $S$ that lie in $P_2$ and $Q_2$ all have negative coefficients (or vice versa).  Since $P \neq Q$, we know that $p$ and $q$ are not scalar multiples of each other.

Now, consider the convex segment $\{\theta(t)=(1-t)p+tq \colon t\in[0,1]\}$ in $\rr^{n}$. Since $P$ and $Q$ are different partitions,  $\theta(t)$ can never have all entries equal to zero.
So, if we move $t$ continuously from $0$ to $1$, $\theta(t)$ moves from $p$ to $q$ in the convex segment. During this process, each time one of the entries of $\theta(t)$ changes sign, say the entry $i$, then the point $s_i$ moves from one part to the other, but the resulting new partition is still a Radon partition.
In this way, every time that the sign of an entry changes, this defines a partition of our sequence, and if for some $t$ the number of entries changing sign is more than one, we define a sequence of Radon partitions changing the respective points one by one in any order.   
\end{proof}

\begin{figure}[h]
\centering
\includegraphics[scale=.12]{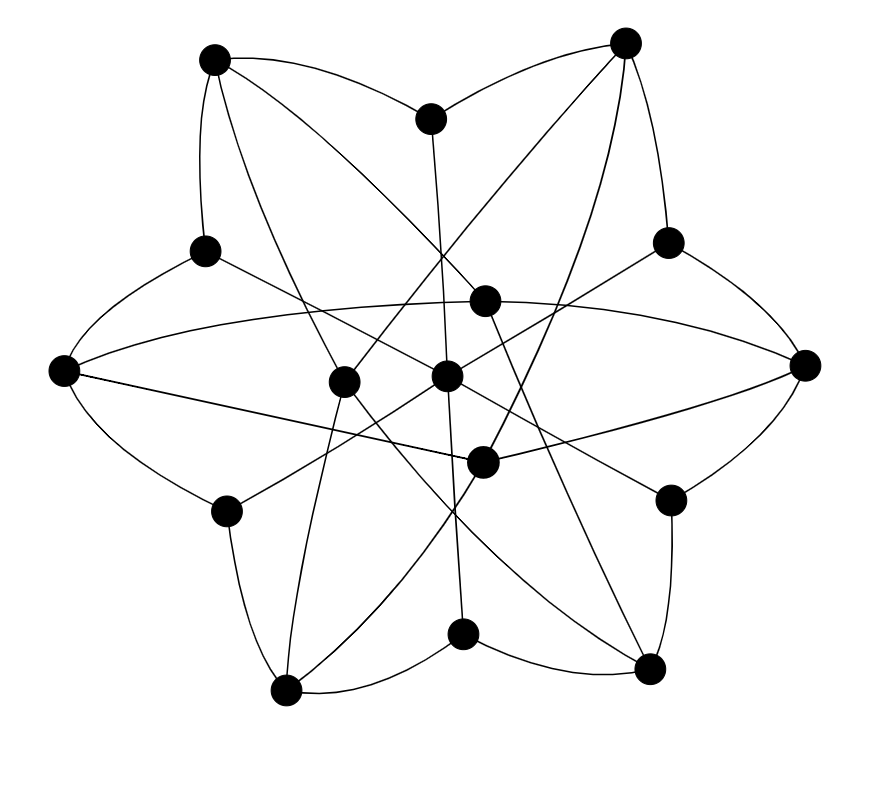}
\hskip 1.2cm
\includegraphics[scale=.16]{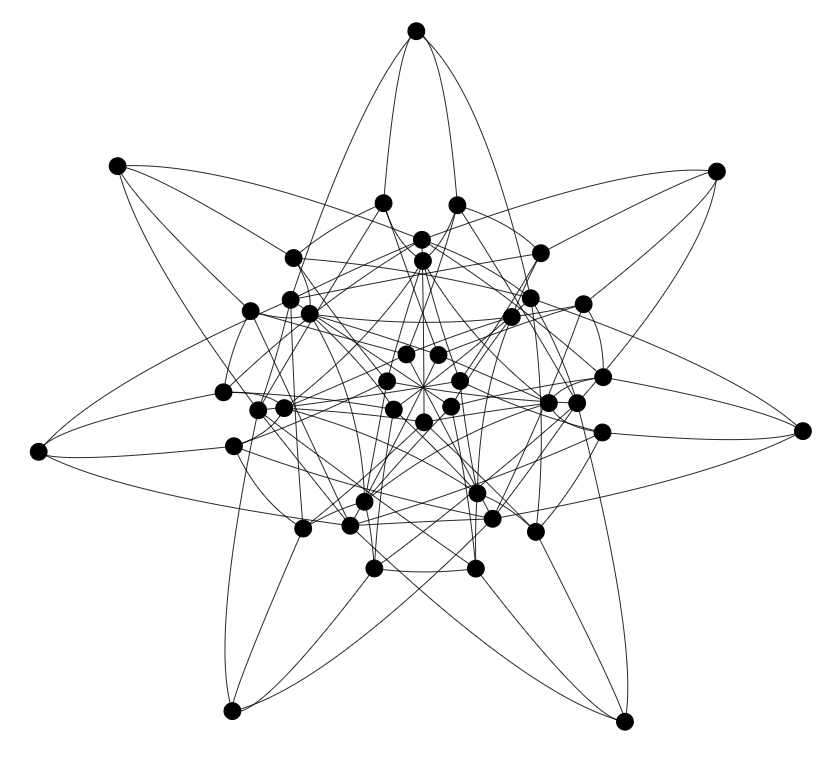}
\caption{Radon partition graphs induced by the vertices of a regular hexagon (left) and a regular heptagon (right).}\label{Radon_graphs}
\end{figure}

%\begin{conjecture}
%Let $S$ be a set of $n\geq \tv(d,r)$ points in $\rr^d$ and $G$ the Tverberg partition graph associated to $S$. Then, for every vertex $v$ in $V(G)$ we conjecture $$(m-1)(n-\tv(d,r)+1)\leq deg(v) \leq (r-1)n.$$ 
%We think that these limits are fair and the maximum is reached for some partition when the points of $S$ are in convex position. We also notice that in each order-type there exists a partition that reaches the minimum bound.   
%\end{conjecture}

The following proposition shows that if $r=2$, the Radon partition graph always has a vertex that attains the minimum possible degree stated in Theorem \ref{degree} in the plane.

\begin{proposition}\label{degreeradonplane}
 Given any set $S\subset \rr^2$ with $n > 4$ points in general position, there exists a Radon partition $P$ of $S$ such that $\deg (P)= n-3$ in the Radon partition graph $G_T[S,2]$.   
\end{proposition}\label{minimum radon}

\begin{proof}
Let $S$ be a set of $n > 4$ points in general position in the plane.  Given a partition $P=P_1,P_2$ of $S$ and a point $q \in P_i$, we will say that $q$ is \emph{essential} for the partition if by moving $q$ to $P_{i+1}$ (where $i+1$ considered modulo 2), we obtain a partition that is not a Radon partition. 

We would like to generate a Radon partition $P= P_1, P_2$ with exactly $3$ essential points.  This would imply $\deg (P)= n- 3$.  
Note that since $S$ is a set of points in general position, the vertices of the convex hull $\conv(S)$ generate a convex $k$ polygon.
Furthermore, for every vertex $w$ on $\conv(S)$ there exists a line $l$ generated by two vertices of $\conv(S)$, say $l=\langle v_1,v_2\rangle$,  that separates $w$ from the rest of the vertices of  $\conv(S)$. 
%as $x_1,x_2,\dots x_k$ with $k\leq n$ 
%in say clockwise order. 
If all points of $S$ are in convex position, or the convex hull of $\conv(w, v_1,v_2)$ has no points of $S$ in its interior, we may set $P_1=\{v_1,v_2\}$ and $P_2 = S\setminus P_1$. Note that $w, v_1,v_2$ are the only essential points for this partition. 
Assume that not all the points are in convex position and that the convex hull of $\{w, v_1,v_2\}$ contains points of $S$ in its interior. Without loss of generality, assume that the convex hull of $w,x_1,x_2$ contains the least number of points of $S$ in its interior. 
Then, with center at $v_1$, rotate $l$ towards $w$ until the first vertex $s_{1}$ of $S$ in the interior of $\conv(S)$ is found. Now with center at $v_2$, rotate $l$ towards $w$ and begin labeling the points in the interior of $\conv(S)$ one by one as $s_{1,1},s_{1,2},\dots s_{1,r}$, 
$0\leq r$, where the last one $s_{1,r}$ satisfies that the line $\Bar{l}$ through $v_2,s_{1,r}$ is the closest to $s_{1}$ but leaves $s_{1}$ and $w$ on one side and the rest of the points 
$v_1,s_{1,1},s_{1,2},\dots s_{1,r-1}$ on the other side. 
Let us partition $S$ as follows: $s_{1,r},v_1,v_2$ and all points are at the same side of $w$ with respect to $\Bar{l}$ except for $s_{1}$ in partition $P_1$ and $s_{1}$ and all the rest of the points are in partition $P_2$.  Observe that $s_{1},v_1,v_2$ are three points that are essential. %
\end{proof}

As we mentioned earlier, we analyzed all the Tverberg $3$-partition graphs induced by at least one representative of each order type up to 10 points in the plane, and more than 2000 graphs of Tverberg $4$-partitions from different order types chosen randomly and uniformly.  Interestingly enough, all graphs like the one shown in Figure \ref{Tv graph},
where all points are the vertices of a regular octagon, satisfy that there is always at least one partition with exactly $(n-Tv(2,r))(r-1)$ neighbors in the Tverberg partition graph, even the vertices of regular hexagons and 10-gons. The following proposition shows that this is always true when the set of points is in strong convex position.

\begin{figure}[H]\label{octagon}
\qquad
\begin{minipage}[c]{0.54\textwidth}%

 \includegraphics[scale=.30]{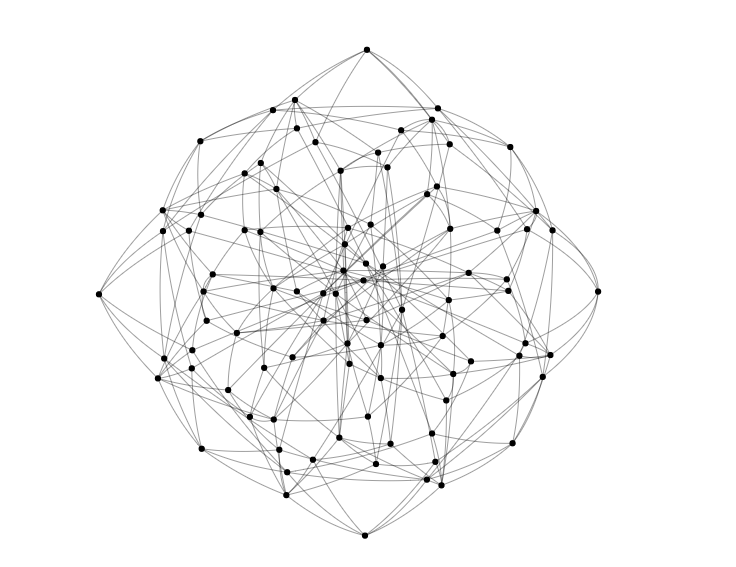}   
\end{minipage}
\centering
\parbox{0.4\textwidth}{
\begin{footnotesize}
\begin{tabular}{cc}
\multicolumn{2}{c}{$|V(G_T[S,3])| =90$}                                                          \\
\multicolumn{2}{c}{$|E(G_T[S,3])|= 272$}                                                         \\
\multicolumn{2}{c}{$\omega(G_T[S,3])= 3$}                                                        \\
\multicolumn{2}{c}{Diameter $= 5$}                                                        \\                                                        \\
\hline
\multicolumn{1}{|c|}{Degree} & \multicolumn{1}{c|}{N° Vertices} \\ \hline
\multicolumn{1}{|c|}{5}                & \multicolumn{1}{c|}{48}                          \\ \hline
\multicolumn{1}{|c|}{6}                & \multicolumn{1}{c|}{16}                          \\ \hline
\multicolumn{1}{|c|}{8}                & \multicolumn{1}{c|}{26}                          \\ \hline
\end{tabular}
\end{footnotesize}
\caption{Tverberg $3$-partition graph $G_T[S,3]=(V,E)$ where $S$ is the set of vertices of a regular octagon. Note that the minimum degree of this graph is $5$, instead of $4$ as suggested in Theorem \ref{degree}. 
\label{Tv graph}
}}
\end{figure}

For the next theorem, we will consider points in strong general convex position, which in the plane simply means that the points are in convex position and there are no three pairs of disjoint points such that the edges induced by these pairs all intersect at a point. (For further details about the definition of strong general position, see \cite{perles2014strong}).

\begin{proposition}\label{convex cases}
Given a set $S$ with $n\geq Tv(2,r)$ points in strong general convex position in the plane, there exists a vertex $v\in V(G_T[S,r])$, and therefore an $r$-partition $v=P$, such that $\deg (v)= \Big(n+1 - \tv(2,r) \Big)(r-1)$ in the Tverberg $r$-partition graph $G_T[S,r]$. 
\end{proposition}

\begin{proof}
Let $S=\{s_1,s_2,\dots s_n\}$  be a set of points in strong general position; assume they are arranged in cyclic clockwise order on some closed curve $C$. 
We will construct a specific partition $P$ such that $\deg(P)=(n-\tv(2,r)+1)(r-1)$. 
Consider the subset $S^\prime=\{s_1,s_2,\dots ,s_{Tv(2,r)}\}$ and a Tverberg $r$-partition $P'$ of $S'$. Since the points are in strongly general position, there is only one point $t$ at the intersection of the convex hulls of the parts, and every part contains at least two points. Then the pigeonhole principle yields that among the $r$ parts there are two of them with exactly two points, say $R=\{r_1,r_2\}$ and $V=\{v_1,v_2\}$, whereas the remaining parts, say, $P_{1},P_{2},\dots,P_{r-2}$, have exactly three points. Let $A=P_{r-2}=\{a_1,a_2,a_3\}$ be one of these parts with three elements chosen arbitrarily. Note that there are exactly five combinatorially different ways of arranging the sets $A,V$ and $R$: $1)$ $r_1a_1v_1a_2r_2a_3v_2$, $2)$ $a_1r_1v_1a_2r_2v_2a_3$, $3)$ $a_1v_1r_1a_2a_3v_2r_2$, $4)$ $a_1v_ 1r_1a_2v_2a_3r_2$, and $5)$ $v_1a_1r_1v_2a_2a_3r_2$ (taken in the clockwise direction).
We will say that a point of a Tverberg $r$-partition of a set is essential if when it is moved to another part, the resulting partition is no longer a Tverberg partition. 

We show that there exists a precise partition $P= P_1, P_2,\dots,P_{r-3}, A,R,V $ of $S$ where all points  in $S\setminus S'$ will be assigned  to $A, R$ or $V$ in such a way that $P$ has the desired degree.    
Note that every point in $P_i\notin \{A,V,R\}$ is essential because $t$ is contained in the interior of its convex hull.\\

\noindent\emph{Case 1:}
Consider the line passing through $v_1$ and $\overline{a_1a_3} \cap \overline{r_1r_2}$ that intersects the circumference $C$ at
a point, say $x$, the line passing through $r_2$ and $\overline{a_1a_3}\cap \overline{v_1 v_2}$ intersects $C$ at point $y$, and the line through $a_2$ and $t$ intersects $C$ at point $z$.  The points $x,y$ and $z$ appear after the points of $S'$ (in cyclic order). Regardless of the order of points $x,y$ and $z$, we notice the following: if there are points of $S\setminus S'$ on arc $\widehat{yr_1}$, we assign each point of $S\setminus S'$ to the part $R$, and observe that with the exception of $r_1$, every point in $S'$ is an essential point. Then $r_1$ may be moved to any of the remaining $r-1$ parts. Moreover, every point in $S\setminus S'$ is not essential either and thus can be assigned to any of the remaining $r-1$ parts as well. Then $P$ satisfies that $\deg(P)=(n-Tv(2,r)+1)(r-1)$. \\

\begin{figure}[h!]
\centering
\input{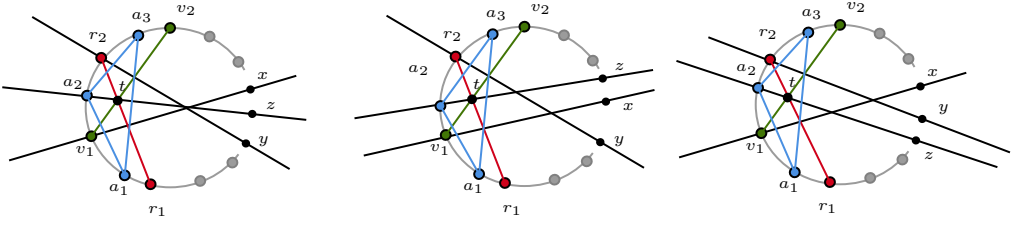}
\caption{All the combinatorical possibilities if there are no points of $S\setminus S'$ on the arc $\widehat{yr_1}$ or on $\widehat{v_2x}$. }\label{conv cases}
\end{figure}

Then we may assume that there are no points of  $S\setminus S'$ on arc $\widehat{yr_1}$. Similarly, if there are points of  $S\setminus S'$ on arc $\widehat{v_2x}$ regardless of the position of $x,y$ and $z$, we may assign every point in $S\setminus S'$ to part $V$, and then $v_2$ becomes the only non-essential point of $P'$ and with a similar argument as before, the desired degree is obtained.\\ 
If there are no points of $S\setminus S'$ on the arc $\widehat{yr_1}$ or on $\widehat{v_2x}$, the points $x$ and $y$ appear in the order $xy$, and every point of $S\setminus S'$ is contained in the arc $\widehat{xy}$. Thus, the only possible orders for the points $x,y$ and $z$ are $xzy$, $zxy$, and $xyz$.\\
If the points appear in the order $xzy$, then, if there are points $S\setminus S'$ on arc $\widehat{xz}$, assigning them to $A$ and the remaining ones (if any) to $R$, we have that $a_3$ is the only non-essential point for $P'$. On the other hand, if all the points of $S\setminus S^\prime$ lie in $\widehat{zy}$, by assigning them to $A$ we have that $a_1$ is the only non-essential point for $P'$. \\ 
If the points appear in the order $zxy$, we assign the points of $S\setminus S'$ to $A$. Thus, the only non-essential point of $P'$ is $a_1$.\\
Similarly, if we have the orientation $xyz$, we assign the points of $S\setminus S'$ to $A$. So $a_3$ is the only non-essential point of $P'$.\\

\noindent\emph{Case 2:}
The line through $v_1$ and $\overline{a_1 a_2} \cap \overline{r_1 r_2}$ intersects $C$ at point $x$, the line through $r_2$ and $\overline{a_2 a_3} \cap \overline{v_1 v_2}$ intersects $C$ at point $y$, and the line through $a_2$ and $t$ intersects $C$ at $z$.\\
If there are points of $S\setminus S'$ on arc $\widehat{a_3x}$, then by the assignment of all points of $S\setminus S'$ to $V$, point $v_2$ becomes the only non-essential point of $P'$.\\
If there are no points of $S\setminus S'$ on arc $\widehat{a_3x}$, but there are on arc $\widehat{ya_1}$, then by the assignment of all points of $S\setminus S'$ to $R$, point $r_1$ becomes the only non-essential point of $P'$.\\
If there are no points of $S\setminus S'$ on the arc $\widehat{a_3x}$ or on $\widehat{ya_1}$, the points $x$ and $y$ appear in the order $xy$, and every point of $S\setminus S'$ is contained in the arc $\widehat{xy}$. Thus, the only possible orders for the points $x,y$ and $z$ are $xzy$, $zxy$ and $xyz$. \\
If we have the order $xzy$ and there are points of $S\setminus S'$ on arc $\widehat{xz}$, assigning those points to $A$ and the remaining ones (if any) to $R$, we have that $a_3$ is the only non-essential point for $P'$. On the other hand, if all the points of $S\setminus S'$ are on arc $\widehat{zy}$ when we assign them to $A$ we have that $a_1$ is the only non-essential point for $P'$. \\
If the orientation is $zxy$, we assign the points of $S\setminus S'$ to $A$. Thus, the only non-essential point of $P'$ is $a_1$. \\
Similarly, if we have the orientation $xyz$, we assign the points of $S\setminus S'$ to $A$. So $a_3$ is the only non-essential point of $P'$.\\

\noindent\emph{Case 3:} 
The line passing through $r_1$ and $\overline{a_1 a_2} \cap \overline{v_1 v_2}$ intersects $C$ at the point $x$, while the line passing through $a_2$ and $t$ cuts $C$ at the point $y$.\\
If there are points of $S\setminus S'$ on arc $\widehat{r_2x}$, then if all the points of $S\setminus S'$ are assigned to $R$, the only non-essential point of $P'$ is $r_2$.\\
The points $x$ and $y$ appear after the points of $S'$ and can be ordered in two different ways, $xy$ or $yx$.
If the order is $xy$ and there are no points of $S\setminus S'$ on arc $\widehat{r_2x}$, but there are on arc $\widehat{ya_1}$, we assign these points to $A$, and the remaining points (if any) to $V$. Thus, the only non-essential point of $P'$ is $a_1$.\\
On the other hand, if the order is $yx$, every point of $S\setminus S'$ is necessarily contained in $\widehat{xa_1}$. Then, assigning all these points to $A$, we have that $a_1$ becomes the only non-essential point of $P'$.\\

\noindent\emph{Case 4:} 
The line passing through $r_1$ and $\overline{v_1v_2}\cap \overline{a_1a_2}$ intersects $C$ at the point $x$, while the line passing through $a_2$ and $t$ intersects $C$ at the point $y$.\\
If there are points of $S\setminus S'$ on the arc $\widehat{r_2x}$, we assign every point of $S\setminus S'$ to $R$. Then $r_2$ is the only non-essential point of $P'$.\\
Suppose then that there are no points of $S\setminus S'$ on $\widehat{r_2x}$. The points $x$ and $y$ appear after the points of $S'$ and can be ordered in two different ways, $xy$ or $yx$.\\
If the order is $xy$ and there are points from $S\setminus S'$ on arc $\widehat{ya_1}$, assigning these points to the part $B$ and the points (if any) contained on arc $\widehat{xy}$ to $R$, we have that $a_1$ becomes the only non-essential point of $P'$.\\
On the other hand, if the order is $yx$, all points in $S\setminus S'$ are contained on arc $\widehat{xa_1}$. Then, assigning all the points to the part $B$, we have that $a_1$ becomes the only non-essential point of $P'$.\\

\noindent\emph{Case 5:} 
Assigning the points of $S\setminus S'$  to the part $A$, we have that $a_3$ is the only non-essential point of $P'$.
\end{proof}

Based on these two propositions and the computational experiments, we believe that the lower bound on Theorem \ref{degree} can always be achieved if $S$ is in strict general position. That is, there always exists a Tverberg partition $P$ such that $P$ is at distance exactly one from another  $\Big(n+1 - \tv(d,r) \Big)(r-1)$ Tverberg partition.

\section{Remarks and open problems}\label{sec:remarks}

In this paper, we have defined and studied the structural properties of the Tverberg partition graph of a given set $S \subset \rr^d$, even when $|S| > \tv(d,r)$.   We study the maximum and minimum degree, the diameter, the clique number, and the connectivity of this graph, and we show that the behavior of the minimum degree for the Tverberg partition graph changes significantly when $|S|$ is large enough. 

Tverberg's theorem guarantees that a set with sufficiently many points $|S|\geq \tv(d,r)$ in
$\rr^d$ can always be partitioned into $r$ parts such that the $(r-1)$-simplex is the nerve complex, describing the intersection pattern of the convex hulls of the parts. In \cite{TverbergTypeTheorems}, \cite{TonoDeb} and \cite{florianyAmsi}, the authors explore 
how other simplicial complexes other than the simplex have ``a Tverberg number" $\tv(d,\mathcal{K})$ for some given simplicial complex $\mathcal{K}$ as well. That is, once a set $S$ has sufficiently many points ($|S|\geq \tv(d,\mathcal{K}))$, there always exists a partition such that its nerve complex is equal to $\mathcal{K}$.  

It becomes very interesting, then, to analyze other Tverberg type graphs where other simplicial complexes are considered, and study their basic structural properties too. For instance, for large $S$, almost all partitions of $S$ into $r$ parts are Tverberg partitions \cite{Soberon:2018gn}. When other simplicial complexes are considered, what can we say about the probability of obtaining other Tverberg graphs of simplicial complexes such as trees or cycles?

Among other interesting open questions we highlight Sierksma's conjecture.  Determining the number of vertices of the Tverberg graph when $|S|\geq (d+1)(r-1)+1$ is an engaging problem. After analyzing thousands of these graphs for sets in the plane, we find it likely that the graphs induced by sets of points that minimize the rectilinear crossing number generate the fewest number of Tverberg partitions when compared with every other order type. In contrast, Tverberg graphs induced by the sets of points that are vertices of a regular polygon generate the largest number of vertices in comparison with other order types (see the right-hand column in the tables below).

The following two tables give the number of vertices of Tverberg partition graphs of all possible simplicial complexes that occur on three vertices. The first table deals with sets of $n$ points $7\leq n\leq 13$ that minimize the number of rectilinear crossings, and the second is for those that are vertices of regular polygons.

%As we mentioned before, the computational experiments were of vital importance during the investigation. These allowed us to visualize some characteristics about the Tverberg partition graphs that are not very intuitive.

%We analyze the graphs of Radon and $3$-Tverberg partitions induced by at least one representative of each order type for sets of $n\in\{7,8,9,10\}$ points. In addition to generating 2000 graphs of $4$-Tverberg partitions belonging to different types of order chosen randomly and uniformly. We particularly note the following.

%\begin{itemize}
   % \item  All the graphs obtained are connected.
      
    %\item  In all graphs obtained, there is always at least one partition with exactly $(n-Tv(2,r))(r-1)$ neighbors in the Tverberg partition graph.
      
    %\item  The graphs induced by sets of $n$ points that minimize the rectilinear crossing number generate the fewest number of vertices in comparison with the representatives of the other order types.
      
    %\item The graphs induced by the sets of $n$ points that are vertices of a regular polygon generate the largest number of vertices compared to the representatives of the other order types.
%\end{itemize}

\begin{table}[H]\label{helices}
\begin{tabular}{|cccccc|}
\hline
\multicolumn{6}{|c|}{\textbf{\begin{tabular}[c]{@{}c@{}}3-Partitions of sets of \boldmath{$n$} points\\ that minimize the rectilinear crossing number\end{tabular}}} \\ \hline
\multicolumn{1}{|c|}{\boldmath{$n$}}  & \multicolumn{1}{c|}{ \tikzset{every picture/.style={line width=0.75pt}} %set default line width to 0.75pt        

\begin{tikzpicture}[x=0.35pt,y=0.35pt,yscale=-1,xscale=1]
%uncomment if require: \path (0,101); %set diagram left start at 0, and has height of 101

%Shape: Ellipse [id:dp520627109009397] 
\draw  [color={rgb, 255:red, 0; green, 0; blue, 0 }  ,draw opacity=1 ][fill={rgb, 255:red, 0; green, 0; blue, 0 }  ,fill opacity=1 ][line width=0.75]  (221.24,84.49) .. controls (221.23,82.84) and (222.54,81.49) .. (224.17,81.49) .. controls (225.79,81.48) and (227.11,82.82) .. (227.11,84.48) .. controls (227.12,86.13) and (225.81,87.48) .. (224.18,87.49) .. controls (222.56,87.49) and (221.24,86.15) .. (221.24,84.49) -- cycle ;
%Shape: Ellipse [id:dp741226136487642] 
\draw  [color={rgb, 255:red, 0; green, 0; blue, 0 }  ,draw opacity=1 ][fill={rgb, 255:red, 0; green, 0; blue, 0 }  ,fill opacity=1 ][line width=0.75]  (256.24,25.67) .. controls (256.23,24.02) and (257.54,22.67) .. (259.17,22.67) .. controls (260.79,22.66) and (262.11,24) .. (262.11,25.66) .. controls (262.12,27.31) and (260.81,28.66) .. (259.18,28.67) .. controls (257.56,28.67) and (256.24,27.33) .. (256.24,25.67) -- cycle ;
%Shape: Ellipse [id:dp4959531575682292] 
\draw  [color={rgb, 255:red, 0; green, 0; blue, 0 }  ,draw opacity=1 ][fill={rgb, 255:red, 0; green, 0; blue, 0 }  ,fill opacity=1 ][line width=0.75]  (291.24,84.49) .. controls (291.23,82.84) and (292.54,81.49) .. (294.17,81.49) .. controls (295.79,81.48) and (297.11,82.82) .. (297.11,84.48) .. controls (297.12,86.13) and (295.81,87.48) .. (294.18,87.49) .. controls (292.56,87.49) and (291.24,86.15) .. (291.24,84.49) -- cycle ;
%Shape: Rectangle [id:dp8260445683735658] 
\draw  [color={rgb, 255:red, 0; green, 0; blue, 0 }  ,draw opacity=0 ] (215.74,15.89) -- (302.24,15.89) -- (302.24,95.64) -- (215.74,95.64) -- cycle ;

\end{tikzpicture}}     & \multicolumn{1}{c|}{ \tikzset{every picture/.style={line width=0.75pt}} %set default line width to 0.75pt        

\begin{tikzpicture}[x=0.35pt,y=0.35pt,yscale=-1,xscale=1]
%uncomment if require: \path (0,102); %set diagram left start at 0, and has height of 102

%Shape: Ellipse [id:dp2807292203199796] 
\draw  [color={rgb, 255:red, 0; green, 0; blue, 0 }  ,draw opacity=1 ][fill={rgb, 255:red, 0; green, 0; blue, 0 }  ,fill opacity=1 ][line width=0.75]  (222.67,78.71) .. controls (222.66,77.05) and (223.97,75.71) .. (225.6,75.7) .. controls (227.22,75.7) and (228.54,77.04) .. (228.54,78.69) .. controls (228.55,80.35) and (227.23,81.7) .. (225.61,81.7) .. controls (223.99,81.7) and (222.67,80.37) .. (222.67,78.71) -- cycle ;
%Shape: Ellipse [id:dp3863113569938246] 
\draw  [color={rgb, 255:red, 0; green, 0; blue, 0 }  ,draw opacity=1 ][fill={rgb, 255:red, 0; green, 0; blue, 0 }  ,fill opacity=1 ][line width=0.75]  (257.67,19.89) .. controls (257.66,18.23) and (258.97,16.89) .. (260.6,16.88) .. controls (262.22,16.88) and (263.54,18.22) .. (263.54,19.87) .. controls (263.55,21.53) and (262.23,22.88) .. (260.61,22.88) .. controls (258.99,22.88) and (257.67,21.55) .. (257.67,19.89) -- cycle ;
%Shape: Ellipse [id:dp841252082210834] 
\draw  [color={rgb, 255:red, 0; green, 0; blue, 0 }  ,draw opacity=1 ][fill={rgb, 255:red, 0; green, 0; blue, 0 }  ,fill opacity=1 ][line width=0.75]  (292.67,78.71) .. controls (292.66,77.05) and (293.97,75.71) .. (295.6,75.7) .. controls (297.22,75.7) and (298.54,77.04) .. (298.54,78.69) .. controls (298.55,80.35) and (297.23,81.7) .. (295.61,81.7) .. controls (293.99,81.7) and (292.67,80.37) .. (292.67,78.71) -- cycle ;
%Straight Lines [id:da807148513893629] 
\draw    (260.6,19.88) -- (225.6,78.7) ;

\end{tikzpicture}}      & \multicolumn{1}{c|}{\tikzset{every picture/.style={line width=0.75pt}} %set default line width to 0.75pt        

\begin{tikzpicture}[x=0.35pt,y=0.35pt,yscale=-1,xscale=1]
%uncomment if require: \path (0,101); %set diagram left start at 0, and has height of 101

%Shape: Ellipse [id:dp6638928546425376] 
\draw  [color={rgb, 255:red, 0; green, 0; blue, 0 }  ,draw opacity=1 ][fill={rgb, 255:red, 0; green, 0; blue, 0 }  ,fill opacity=1 ][line width=0.75]  (221.17,82.21) .. controls (221.16,80.55) and (222.47,79.21) .. (224.1,79.2) .. controls (225.72,79.2) and (227.04,80.54) .. (227.04,82.19) .. controls (227.05,83.85) and (225.73,85.2) .. (224.11,85.2) .. controls (222.49,85.2) and (221.17,83.87) .. (221.17,82.21) -- cycle ;
%Shape: Ellipse [id:dp28714637769171225] 
\draw  [color={rgb, 255:red, 0; green, 0; blue, 0 }  ,draw opacity=1 ][fill={rgb, 255:red, 0; green, 0; blue, 0 }  ,fill opacity=1 ][line width=0.75]  (256.17,23.39) .. controls (256.16,21.73) and (257.47,20.39) .. (259.1,20.38) .. controls (260.72,20.38) and (262.04,21.72) .. (262.04,23.37) .. controls (262.05,25.03) and (260.73,26.38) .. (259.11,26.38) .. controls (257.49,26.38) and (256.17,25.05) .. (256.17,23.39) -- cycle ;
%Shape: Ellipse [id:dp3530230196943922] 
\draw  [color={rgb, 255:red, 0; green, 0; blue, 0 }  ,draw opacity=1 ][fill={rgb, 255:red, 0; green, 0; blue, 0 }  ,fill opacity=1 ][line width=0.75]  (291.17,82.21) .. controls (291.16,80.55) and (292.47,79.21) .. (294.1,79.2) .. controls (295.72,79.2) and (297.04,80.54) .. (297.04,82.19) .. controls (297.05,83.85) and (295.73,85.2) .. (294.11,85.2) .. controls (292.49,85.2) and (291.17,83.87) .. (291.17,82.21) -- cycle ;
%Straight Lines [id:da7113663730011199] 
\draw    (259.1,23.38) -- (224.1,82.2) ;
%Straight Lines [id:da5719389310857876] 
\draw    (259.1,23.38) -- (294.1,82.2) ;

\end{tikzpicture}}      & \multicolumn{1}{c|}{\tikzset{every picture/.style={line width=0.75pt}} %set default line width to 0.75pt        

\begin{tikzpicture}[x=0.35pt,y=0.35pt,yscale=-1,xscale=1]
%uncomment if require: \path (0,99); %set diagram left start at 0, and has height of 99

%Shape: Ellipse [id:dp006349359549263189] 
\draw  [color={rgb, 255:red, 0; green, 0; blue, 0 }  ,draw opacity=1 ][fill={rgb, 255:red, 0; green, 0; blue, 0 }  ,fill opacity=1 ][line width=0.75]  (222.67,78.71) .. controls (222.66,77.05) and (223.97,75.71) .. (225.6,75.7) .. controls (227.22,75.7) and (228.54,77.04) .. (228.54,78.69) .. controls (228.55,80.35) and (227.23,81.7) .. (225.61,81.7) .. controls (223.99,81.7) and (222.67,80.37) .. (222.67,78.71) -- cycle ;
%Shape: Ellipse [id:dp5799430512543833] 
\draw  [color={rgb, 255:red, 0; green, 0; blue, 0 }  ,draw opacity=1 ][fill={rgb, 255:red, 0; green, 0; blue, 0 }  ,fill opacity=1 ][line width=0.75]  (257.67,19.89) .. controls (257.66,18.23) and (258.97,16.89) .. (260.6,16.88) .. controls (262.22,16.88) and (263.54,18.22) .. (263.54,19.87) .. controls (263.55,21.53) and (262.23,22.88) .. (260.61,22.88) .. controls (258.99,22.88) and (257.67,21.55) .. (257.67,19.89) -- cycle ;
%Shape: Ellipse [id:dp9662105303994031] 
\draw  [color={rgb, 255:red, 0; green, 0; blue, 0 }  ,draw opacity=1 ][fill={rgb, 255:red, 0; green, 0; blue, 0 }  ,fill opacity=1 ][line width=0.75]  (292.67,78.71) .. controls (292.66,77.05) and (293.97,75.71) .. (295.6,75.7) .. controls (297.22,75.7) and (298.54,77.04) .. (298.54,78.69) .. controls (298.55,80.35) and (297.23,81.7) .. (295.61,81.7) .. controls (293.99,81.7) and (292.67,80.37) .. (292.67,78.71) -- cycle ;
%Straight Lines [id:da020171212977334374] 
\draw    (260.6,19.88) -- (225.6,78.7) ;
%Straight Lines [id:da9446219222784662] 
\draw    (260.6,19.88) -- (295.6,78.7) ;
%Straight Lines [id:da6422206538621122] 
\draw    (295.6,78.7) -- (225.6,78.7) ;

\end{tikzpicture}}     &      \multicolumn{1}{c|}{\tikzset{every picture/.style={line width=0.75pt}} %set default line width to 0.75pt        

\begin{tikzpicture}[x=0.35pt,y=0.35pt,yscale=-1,xscale=1]
%uncomment if require: \path (0,109); %set diagram left start at 0, and has height of 109

%Shape: Ellipse [id:dp6401523548037804] 
\draw  [color={rgb, 255:red, 0; green, 0; blue, 0 }  ,draw opacity=1 ][fill={rgb, 255:red, 0; green, 0; blue, 0 }  ,fill opacity=1 ][line width=0.75]  (222.24,86.49) .. controls (222.23,84.84) and (223.54,83.49) .. (225.17,83.49) .. controls (226.79,83.48) and (228.11,84.82) .. (228.11,86.48) .. controls (228.12,88.13) and (226.81,89.48) .. (225.18,89.49) .. controls (223.56,89.49) and (222.24,88.15) .. (222.24,86.49) -- cycle ;
%Shape: Ellipse [id:dp34173650396969135] 
\draw  [color={rgb, 255:red, 0; green, 0; blue, 0 }  ,draw opacity=1 ][fill={rgb, 255:red, 0; green, 0; blue, 0 }  ,fill opacity=1 ][line width=0.75]  (257.24,27.67) .. controls (257.23,26.02) and (258.54,24.67) .. (260.17,24.67) .. controls (261.79,24.66) and (263.11,26) .. (263.11,27.66) .. controls (263.12,29.31) and (261.81,30.66) .. (260.18,30.67) .. controls (258.56,30.67) and (257.24,29.33) .. (257.24,27.67) -- cycle ;
%Shape: Ellipse [id:dp6834161702593888] 
\draw  [color={rgb, 255:red, 0; green, 0; blue, 0 }  ,draw opacity=1 ][fill={rgb, 255:red, 0; green, 0; blue, 0 }  ,fill opacity=1 ][line width=0.75]  (292.24,86.49) .. controls (292.23,84.84) and (293.54,83.49) .. (295.17,83.49) .. controls (296.79,83.48) and (298.11,84.82) .. (298.11,86.48) .. controls (298.12,88.13) and (296.81,89.48) .. (295.18,89.49) .. controls (293.56,89.49) and (292.24,88.15) .. (292.24,86.49) -- cycle ;
%Shape: Polygon [id:ds850754514458737] 
\draw  [fill={rgb, 255:red, 0; green, 0; blue, 0 }  ,fill opacity=0.59 ] (260.18,27.67) -- (295.18,86.49) -- (225.18,86.49) -- cycle ;

\end{tikzpicture}}  \\ \hline
\multicolumn{1}{|c|}{\textbf{7}}  & \multicolumn{1}{c|}{127}  & \multicolumn{1}{c|}{100}   & \multicolumn{1}{c|}{70}    & \multicolumn{1}{c|}{0}    & \multicolumn{1}{c|}{4}       \\ \hline
\multicolumn{1}{|c|}{\textbf{8}}  & \multicolumn{1}{c|}{253}  & \multicolumn{1}{c|}{344}   & \multicolumn{1}{c|}{317}   & \multicolumn{1}{c|}{4}    & \multicolumn{1}{c|}{48}     \\ \hline
\multicolumn{1}{|c|}{\textbf{9}}  & \multicolumn{1}{c|}{466}  & \multicolumn{1}{c|}{977}   & \multicolumn{1}{c|}{1220}  & \multicolumn{1}{c|}{26}   & \multicolumn{1}{c|}{336}    \\ \hline
\multicolumn{1}{|c|}{\textbf{10}} & \multicolumn{1}{c|}{800}  & \multicolumn{1}{c|}{2515}  & \multicolumn{1}{c|}{4045}  & \multicolumn{1}{c|}{170}  & \multicolumn{1}{c|}{1800}   \\ \hline
\multicolumn{1}{|c|}{\textbf{11}} & \multicolumn{1}{c|}{1301} & \multicolumn{1}{c|}{5946}  & \multicolumn{1}{c|}{12310} & \multicolumn{1}{c|}{627}  & \multicolumn{1}{c|}{8317}   \\ \hline
\multicolumn{1}{|c|}{\textbf{12}} & \multicolumn{1}{c|}{2052} & \multicolumn{1}{c|}{13182} & \multicolumn{1}{c|}{34998} & \multicolumn{1}{c|}{2814} & \multicolumn{1}{c|}{33480}  \\ \hline
\multicolumn{1}{|c|}{\textbf{13}} & \multicolumn{1}{c|}{3070} & \multicolumn{1}{c|}{28265} & \multicolumn{1}{c|}{93926} & \multicolumn{1}{c|}{8766} & \multicolumn{1}{c|}{127598} \\ \hline
\end{tabular}
\end{table}

\begin{table}[H] \label{poligonos}
\begin{tabular}{|cccccc|}
\hline
\multicolumn{6}{|c|}{\textbf{\begin{tabular}[c]{@{}c@{}}3-Partitions of sets of \boldmath{$n$} points\\ that are the vertices of regular polygons\end{tabular}}} \\ \hline
\multicolumn{1}{|c|}{\boldmath{$n$}}  & \multicolumn{1}{c|}{\tikzset{every picture/.style={line width=0.75pt}} %set default line width to 0.75pt        

\begin{tikzpicture}[x=0.35pt,y=0.35pt,yscale=-1,xscale=1]
%uncomment if require: \path (0,101); %set diagram left start at 0, and has height of 101

%Shape: Ellipse [id:dp520627109009397] 
\draw  [color={rgb, 255:red, 0; green, 0; blue, 0 }  ,draw opacity=1 ][fill={rgb, 255:red, 0; green, 0; blue, 0 }  ,fill opacity=1 ][line width=0.75]  (221.24,84.49) .. controls (221.23,82.84) and (222.54,81.49) .. (224.17,81.49) .. controls (225.79,81.48) and (227.11,82.82) .. (227.11,84.48) .. controls (227.12,86.13) and (225.81,87.48) .. (224.18,87.49) .. controls (222.56,87.49) and (221.24,86.15) .. (221.24,84.49) -- cycle ;
%Shape: Ellipse [id:dp741226136487642] 
\draw  [color={rgb, 255:red, 0; green, 0; blue, 0 }  ,draw opacity=1 ][fill={rgb, 255:red, 0; green, 0; blue, 0 }  ,fill opacity=1 ][line width=0.75]  (256.24,25.67) .. controls (256.23,24.02) and (257.54,22.67) .. (259.17,22.67) .. controls (260.79,22.66) and (262.11,24) .. (262.11,25.66) .. controls (262.12,27.31) and (260.81,28.66) .. (259.18,28.67) .. controls (257.56,28.67) and (256.24,27.33) .. (256.24,25.67) -- cycle ;
%Shape: Ellipse [id:dp4959531575682292] 
\draw  [color={rgb, 255:red, 0; green, 0; blue, 0 }  ,draw opacity=1 ][fill={rgb, 255:red, 0; green, 0; blue, 0 }  ,fill opacity=1 ][line width=0.75]  (291.24,84.49) .. controls (291.23,82.84) and (292.54,81.49) .. (294.17,81.49) .. controls (295.79,81.48) and (297.11,82.82) .. (297.11,84.48) .. controls (297.12,86.13) and (295.81,87.48) .. (294.18,87.49) .. controls (292.56,87.49) and (291.24,86.15) .. (291.24,84.49) -- cycle ;
%Shape: Rectangle [id:dp8260445683735658] 
\draw  [color={rgb, 255:red, 0; green, 0; blue, 0 }  ,draw opacity=0 ] (215.74,15.89) -- (302.24,15.89) -- (302.24,95.64) -- (215.74,95.64) -- cycle ;

\end{tikzpicture}}     & \multicolumn{1}{c|}{\tikzset{every picture/.style={line width=0.75pt}} %set default line width to 0.75pt        

\begin{tikzpicture}[x=0.35pt,y=0.35pt,yscale=-1,xscale=1]
%uncomment if require: \path (0,102); %set diagram left start at 0, and has height of 102

%Shape: Ellipse [id:dp2807292203199796] 
\draw  [color={rgb, 255:red, 0; green, 0; blue, 0 }  ,draw opacity=1 ][fill={rgb, 255:red, 0; green, 0; blue, 0 }  ,fill opacity=1 ][line width=0.75]  (222.67,78.71) .. controls (222.66,77.05) and (223.97,75.71) .. (225.6,75.7) .. controls (227.22,75.7) and (228.54,77.04) .. (228.54,78.69) .. controls (228.55,80.35) and (227.23,81.7) .. (225.61,81.7) .. controls (223.99,81.7) and (222.67,80.37) .. (222.67,78.71) -- cycle ;
%Shape: Ellipse [id:dp3863113569938246] 
\draw  [color={rgb, 255:red, 0; green, 0; blue, 0 }  ,draw opacity=1 ][fill={rgb, 255:red, 0; green, 0; blue, 0 }  ,fill opacity=1 ][line width=0.75]  (257.67,19.89) .. controls (257.66,18.23) and (258.97,16.89) .. (260.6,16.88) .. controls (262.22,16.88) and (263.54,18.22) .. (263.54,19.87) .. controls (263.55,21.53) and (262.23,22.88) .. (260.61,22.88) .. controls (258.99,22.88) and (257.67,21.55) .. (257.67,19.89) -- cycle ;
%Shape: Ellipse [id:dp841252082210834] 
\draw  [color={rgb, 255:red, 0; green, 0; blue, 0 }  ,draw opacity=1 ][fill={rgb, 255:red, 0; green, 0; blue, 0 }  ,fill opacity=1 ][line width=0.75]  (292.67,78.71) .. controls (292.66,77.05) and (293.97,75.71) .. (295.6,75.7) .. controls (297.22,75.7) and (298.54,77.04) .. (298.54,78.69) .. controls (298.55,80.35) and (297.23,81.7) .. (295.61,81.7) .. controls (293.99,81.7) and (292.67,80.37) .. (292.67,78.71) -- cycle ;
%Straight Lines [id:da807148513893629] 
\draw    (260.6,19.88) -- (225.6,78.7) ;

\end{tikzpicture}}      & \multicolumn{1}{c|}{\tikzset{every picture/.style={line width=0.75pt}} %set default line width to 0.75pt        

\begin{tikzpicture}[x=0.35pt,y=0.35pt,yscale=-1,xscale=1]
%uncomment if require: \path (0,101); %set diagram left start at 0, and has height of 101

%Shape: Ellipse [id:dp6638928546425376] 
\draw  [color={rgb, 255:red, 0; green, 0; blue, 0 }  ,draw opacity=1 ][fill={rgb, 255:red, 0; green, 0; blue, 0 }  ,fill opacity=1 ][line width=0.75]  (221.17,82.21) .. controls (221.16,80.55) and (222.47,79.21) .. (224.1,79.2) .. controls (225.72,79.2) and (227.04,80.54) .. (227.04,82.19) .. controls (227.05,83.85) and (225.73,85.2) .. (224.11,85.2) .. controls (222.49,85.2) and (221.17,83.87) .. (221.17,82.21) -- cycle ;
%Shape: Ellipse [id:dp28714637769171225] 
\draw  [color={rgb, 255:red, 0; green, 0; blue, 0 }  ,draw opacity=1 ][fill={rgb, 255:red, 0; green, 0; blue, 0 }  ,fill opacity=1 ][line width=0.75]  (256.17,23.39) .. controls (256.16,21.73) and (257.47,20.39) .. (259.1,20.38) .. controls (260.72,20.38) and (262.04,21.72) .. (262.04,23.37) .. controls (262.05,25.03) and (260.73,26.38) .. (259.11,26.38) .. controls (257.49,26.38) and (256.17,25.05) .. (256.17,23.39) -- cycle ;
%Shape: Ellipse [id:dp3530230196943922] 
\draw  [color={rgb, 255:red, 0; green, 0; blue, 0 }  ,draw opacity=1 ][fill={rgb, 255:red, 0; green, 0; blue, 0 }  ,fill opacity=1 ][line width=0.75]  (291.17,82.21) .. controls (291.16,80.55) and (292.47,79.21) .. (294.1,79.2) .. controls (295.72,79.2) and (297.04,80.54) .. (297.04,82.19) .. controls (297.05,83.85) and (295.73,85.2) .. (294.11,85.2) .. controls (292.49,85.2) and (291.17,83.87) .. (291.17,82.21) -- cycle ;
%Straight Lines [id:da7113663730011199] 
\draw    (259.1,23.38) -- (224.1,82.2) ;
%Straight Lines [id:da5719389310857876] 
\draw    (259.1,23.38) -- (294.1,82.2) ;

\end{tikzpicture}}      & \multicolumn{1}{c|}{\tikzset{every picture/.style={line width=0.75pt}} %set default line width to 0.75pt        

\begin{tikzpicture}[x=0.35pt,y=0.35pt,yscale=-1,xscale=1]
%uncomment if require: \path (0,99); %set diagram left start at 0, and has height of 99

%Shape: Ellipse [id:dp006349359549263189] 
\draw  [color={rgb, 255:red, 0; green, 0; blue, 0 }  ,draw opacity=1 ][fill={rgb, 255:red, 0; green, 0; blue, 0 }  ,fill opacity=1 ][line width=0.75]  (222.67,78.71) .. controls (222.66,77.05) and (223.97,75.71) .. (225.6,75.7) .. controls (227.22,75.7) and (228.54,77.04) .. (228.54,78.69) .. controls (228.55,80.35) and (227.23,81.7) .. (225.61,81.7) .. controls (223.99,81.7) and (222.67,80.37) .. (222.67,78.71) -- cycle ;
%Shape: Ellipse [id:dp5799430512543833] 
\draw  [color={rgb, 255:red, 0; green, 0; blue, 0 }  ,draw opacity=1 ][fill={rgb, 255:red, 0; green, 0; blue, 0 }  ,fill opacity=1 ][line width=0.75]  (257.67,19.89) .. controls (257.66,18.23) and (258.97,16.89) .. (260.6,16.88) .. controls (262.22,16.88) and (263.54,18.22) .. (263.54,19.87) .. controls (263.55,21.53) and (262.23,22.88) .. (260.61,22.88) .. controls (258.99,22.88) and (257.67,21.55) .. (257.67,19.89) -- cycle ;
%Shape: Ellipse [id:dp9662105303994031] 
\draw  [color={rgb, 255:red, 0; green, 0; blue, 0 }  ,draw opacity=1 ][fill={rgb, 255:red, 0; green, 0; blue, 0 }  ,fill opacity=1 ][line width=0.75]  (292.67,78.71) .. controls (292.66,77.05) and (293.97,75.71) .. (295.6,75.7) .. controls (297.22,75.7) and (298.54,77.04) .. (298.54,78.69) .. controls (298.55,80.35) and (297.23,81.7) .. (295.61,81.7) .. controls (293.99,81.7) and (292.67,80.37) .. (292.67,78.71) -- cycle ;
%Straight Lines [id:da020171212977334374] 
\draw    (260.6,19.88) -- (225.6,78.7) ;
%Straight Lines [id:da9446219222784662] 
\draw    (260.6,19.88) -- (295.6,78.7) ;
%Straight Lines [id:da6422206538621122] 
\draw    (295.6,78.7) -- (225.6,78.7) ;

\end{tikzpicture}}     &      \multicolumn{1}{c|}{\tikzset{every picture/.style={line width=0.75pt}} %set default line width to 0.75pt        

\begin{tikzpicture}[x=0.35pt,y=0.35pt,yscale=-1,xscale=1]
%uncomment if require: \path (0,109); %set diagram left start at 0, and has height of 109

%Shape: Ellipse [id:dp6401523548037804] 
\draw  [color={rgb, 255:red, 0; green, 0; blue, 0 }  ,draw opacity=1 ][fill={rgb, 255:red, 0; green, 0; blue, 0 }  ,fill opacity=1 ][line width=0.75]  (222.24,86.49) .. controls (222.23,84.84) and (223.54,83.49) .. (225.17,83.49) .. controls (226.79,83.48) and (228.11,84.82) .. (228.11,86.48) .. controls (228.12,88.13) and (226.81,89.48) .. (225.18,89.49) .. controls (223.56,89.49) and (222.24,88.15) .. (222.24,86.49) -- cycle ;
%Shape: Ellipse [id:dp34173650396969135] 
\draw  [color={rgb, 255:red, 0; green, 0; blue, 0 }  ,draw opacity=1 ][fill={rgb, 255:red, 0; green, 0; blue, 0 }  ,fill opacity=1 ][line width=0.75]  (257.24,27.67) .. controls (257.23,26.02) and (258.54,24.67) .. (260.17,24.67) .. controls (261.79,24.66) and (263.11,26) .. (263.11,27.66) .. controls (263.12,29.31) and (261.81,30.66) .. (260.18,30.67) .. controls (258.56,30.67) and (257.24,29.33) .. (257.24,27.67) -- cycle ;
%Shape: Ellipse [id:dp6834161702593888] 
\draw  [color={rgb, 255:red, 0; green, 0; blue, 0 }  ,draw opacity=1 ][fill={rgb, 255:red, 0; green, 0; blue, 0 }  ,fill opacity=1 ][line width=0.75]  (292.24,86.49) .. controls (292.23,84.84) and (293.54,83.49) .. (295.17,83.49) .. controls (296.79,83.48) and (298.11,84.82) .. (298.11,86.48) .. controls (298.12,88.13) and (296.81,89.48) .. (295.18,89.49) .. controls (293.56,89.49) and (292.24,88.15) .. (292.24,86.49) -- cycle ;
%Shape: Polygon [id:ds850754514458737] 
\draw  [fill={rgb, 255:red, 0; green, 0; blue, 0 }  ,fill opacity=0.59 ] (260.18,27.67) -- (295.18,86.49) -- (225.18,86.49) -- cycle ;

\end{tikzpicture}}  \\ \hline
\multicolumn{1}{|c|}{\textbf{7}}  & \multicolumn{1}{c|}{105}  & \multicolumn{1}{c|}{154}   & \multicolumn{1}{c|}{28}    & \multicolumn{1}{c|}{7}    & 7      \\ \hline
\multicolumn{1}{|c|}{\textbf{8}}  & \multicolumn{1}{c|}{196}  & \multicolumn{1}{c|}{512}   & \multicolumn{1}{c|}{152}   & \multicolumn{1}{c|}{16}   & 90     \\ \hline
\multicolumn{1}{|c|}{\textbf{9}}  & \multicolumn{1}{c|}{336}  & \multicolumn{1}{c|}{1467}  & \multicolumn{1}{c|}{630}   & \multicolumn{1}{c|}{138}  & 454    \\ \hline
\multicolumn{1}{|c|}{\textbf{10}} & \multicolumn{1}{c|}{540}  & \multicolumn{1}{c|}{3820}  & \multicolumn{1}{c|}{2215}  & \multicolumn{1}{c|}{370}  & 2385   \\ \hline
\multicolumn{1}{|c|}{\textbf{11}} & \multicolumn{1}{c|}{825}  & \multicolumn{1}{c|}{9328}  & \multicolumn{1}{c|}{6974}  & \multicolumn{1}{c|}{1419} & 9955   \\ \hline
\multicolumn{1}{|c|}{\textbf{12}} & \multicolumn{1}{c|}{1210} & \multicolumn{1}{c|}{21792} & \multicolumn{1}{c|}{20304} & \multicolumn{1}{c|}{2776} & 40444  \\ \hline
\multicolumn{1}{|c|}{\textbf{13}} & \multicolumn{1}{c|}{1716} & \multicolumn{1}{c|}{49361} & \multicolumn{1}{c|}{55796} & \multicolumn{1}{c|}{9768} & 144984 \\ \hline
\end{tabular}
\end{table}

\section{Acknowledgements}
We are truly grateful to Edgardo Rold\'an, Luis Montejano, and Luerbio Faria, who gave us useful suggestions in the some stages of this project. The research of the first, and fourth author were supported by Proyecto PAPIIT AG100721 UNAM and second author was funded by grants NSF DMS 2054419 and NSF DMS 2237324.  
%\end{acknowledgements}

\bibliographystyle{plain} % We choose the "plain" reference style
\bibliography{bibliography}

\noindent \texttt{doliveros@im.unam.mx}\\
Instituto de Matem\'aticas, UNAM, Juriquilla, Quer\'etaro, M\'exico.\\[0.1cm]

\noindent \texttt{roldan@mis.mpg.de}\\
Max Planck Insitute for the Mathematics in the Sciences and ScaDS.AI Leipzig University, Leipzig, Germany.\\[0.1cm]

\noindent \texttt{psoberon@gc.cuny.edu}\\
The Graduate Center, City University of New York, NY, USA.
Baruch College, City University of New York, NY, USA.\\[0.1cm]

\noindent \texttt{antor@ucdavis.edu}\\
Department of Mathematics, UC. Davis, CA, USA.\\

\end{document}